\documentclass[12pt]{article}
\usepackage{amsmath,enumerate,amsfonts,amssymb,color,graphicx,amsthm}

\usepackage[colorlinks=true, allcolors=black]{hyperref}
\usepackage{todonotes}
\usepackage[normalem]{ulem}
\numberwithin{equation}{section}
\usepackage{cite}
\usepackage{mathtools}
\usepackage[nottoc,notlot,notlof]{tocbibind}
\usepackage[toc,page]{appendix}

\usepackage{tikz}
\usepackage{mathdots}
\usepackage{yhmath}
\usepackage{cancel}
\usepackage{siunitx}
\usepackage{array}
\usepackage{multirow}
\usepackage{gensymb}
\usepackage{tabularx}
\usepackage{extarrows}
\usepackage{booktabs}
\usetikzlibrary{fadings}
\usetikzlibrary{patterns}
\usetikzlibrary{shadows.blur}
\usetikzlibrary{shapes}

\renewenvironment{thebibliography}[1]{
	\begin{oldthebibliography}{#1}
		\setlength{\itemsep}{0.5em}
		\setlength{\parskip}{0em}
	}
	{
	\end{oldthebibliography}
}

\usepackage{lipsum}

\def\XXint#1#2#3{{\setbox0=\hbox{$#1{#2#3}{\int}$ }
		\vcenter{\hbox{$#2#3$ }}\kern-.6\wd0}}

\newlength{\leftstackrelawd}
\newlength{\leftstackrelbwd}
\def\leftstackrel#1#2{\settowidth{\leftstackrelawd}%
	{${{}^{#1}}$}\settowidth{\leftstackrelbwd}{$#2$}%
	\addtolength{\leftstackrelawd}{-\leftstackrelbwd}%
	\leavevmode\ifthenelse{\lengthtest{\leftstackrelawd>0pt}}%
	{\kern-.5\leftstackrelawd}{}\mathrel{\mathop{#2}\limits^{#1}}}

\theoremstyle{plain}
\newtheorem{thm}{Theorem}[section]
\newtheorem{lem}[thm]{Lemma}

\newtheorem{prop}[thm]{Proposition}
\newtheorem*{thm*}{Theorem}

\theoremstyle{definition}

\newtheorem{rmk}[thm]{Remark}
\newtheorem{?}[thm]{Problem}

\newtheorem{innercustomthm}{Theorem}
\newenvironment{customthm}[1]
{\renewcommand\theinnercustomthm{#1}\innercustomthm}
{\endinnercustomthm}

\newcommand{\ep}{\varepsilon}
\renewcommand{\phi}{\varphi}
\renewcommand{\epsilon}{\varepsilon}
\makeatletter
\def\@cite#1#2{[\textbf{#1\if@tempswa , #2\fi}]}
\def\@biblabel#1{[\textbf{#1}]}
\makeatother

\makeatletter
\newcommand*{\defeq}{\mathrel{\rlap{%
			\raisebox{0.3ex}{$\m@th\cdot$}}%
		\raisebox{-0.3ex}{$\m@th\cdot$}}%
	=}
\makeatother

\makeatletter
\newcommand*{\eqdef}{=\mathrel{\rlap{%
			\raisebox{0.3ex}{$\m@th\cdot$}}%
		\raisebox{-0.3ex}{$\m@th\cdot$}}%
	}
\makeatother

\newcounter{marnote}

\makeatletter
\def\underbracex#1#2{\mathop{\vtop{\m@th\ialign{##\crcr
				$\hfil\displaystyle{#2}\hfil$\crcr
				\noalign{\kern3\p@\nointerlineskip}%
				#1\crcr\noalign{\kern3\p@}}}}\limits}

\def\upbracefilla{$\m@th \setbox\z@\hbox{$\braceld$}%
	\bracelu\leaders\vrule \@height\ht\z@ \@depth\z@\hfill 
	\kern\p@\vrule \@width\p@\kern\p@\vrule \@width\p@\kern\p@\vrule \@width\p@
	$}

\def\upbracefillb{$\m@th \setbox\z@\hbox{$\braceld$}%
	\vrule \@width\p@\kern\p@\vrule \@width\p@\kern\p@\vrule \@width\p@\kern\p@
	\leaders\vrule \@height\ht\z@ \@depth\z@\hfill\bracerd
	\braceld\leaders\vrule \@height\ht\z@ \@depth\z@\hfill
	\kern\p@\vrule \@width\p@\kern\p@\vrule \@width\p@\kern\p@\vrule \@width\p@
	$}

\def\upbracefillc{$\m@th \setbox\z@\hbox{$\braceld$}%
	\vrule \@width\p@\kern\p@\vrule \@width\p@\kern\p@\vrule \@width\p@\kern\p@
	\leaders\vrule \@height\ht\z@ \@depth\z@\hfill
	\kern\p@\vrule \@width\p@\kern\p@\vrule \@width\p@\kern\p@\vrule \@width\p@
	$}

\def\upbracefilld{$\m@th \setbox\z@\hbox{$\braceld$}%
	\vrule \@width\p@\kern\p@\vrule \@width\p@\kern\p@\vrule \@width\p@\kern\p@
	\leaders\vrule \@height\ht\z@ \@depth\z@\hfill\braceru$}

\def\upbracefillbd{$\m@th \setbox\z@\hbox{$\braceld$}%
	\vrule \@width\p@\kern\p@\vrule \@width\p@\kern\p@\vrule \@width\p@\kern\p@
	\bracerd\braceld
	\leaders\vrule \@height\ht\z@ \@depth\z@\hfill\braceru$}

\makeatother

\begin{document}
	\title{\vspace{-1.1cm}The $\sigma_k$-Loewner-Nirenberg problem on Riemannian manifolds for $k<\frac{n}{2}$}
	\author{Jonah A. J. Duncan\footnote{Johns Hopkins University, 404 Krieger Hall, Department of Mathematics, 3400 N. Charles Street, Baltimore, MD 21218, US. Email: jdunca33@jhu.edu.} ~and Luc Nguyen\footnote{Mathematical Institute and St Edmund Hall, University of Oxford, Andrew Wiles Building, Radcliffe Observatory Quarter, Woodstock Road, OX2 6GG, UK. Email: luc.nguyen@maths.ox.ac.uk.}}
\maketitle
\vspace*{-7mm}
	\begin{abstract}
		Let $(M^n,g_0)$ be a smooth compact Riemannian manifold of dimension $n\geq 3$ with non-empty boundary $\partial M$. Let $\Gamma\subset\mathbb{R}^n$ be a symmetric convex cone and $f$ a symmetric defining function for $\Gamma$ satisfying standard assumptions. Under an algebraic condition on $\Gamma$, which is satisfied for example by the G\r{a}rding cones $\Gamma_k^+$ when $k<\frac{n}{2}$, we prove the existence of a Lipschitz viscosity solution $g_u = e^{2u}g_0$ to the fully nonlinear Loewner-Nirenberg problem associated to $(f,\Gamma)$,
		\begin{align*}
		\begin{cases} f(\lambda(-g_u^{-1}A_{g_u})) = 1, \quad \lambda(-g_u^{-1}A_{g_u}) \in \Gamma & \mathrm{on~}M\backslash\partial M \\
		u(x)\rightarrow+\infty & \mathrm{as~}\operatorname{dist}_{g_0}(x,\partial M)\rightarrow 0,
		\end{cases} 
		\end{align*}
		where $A_{g_u}$ is the Schouten tensor of $g_u$. Previous results on Euclidean domains show that, in general, $u$ is not differentiable. The solution $u$ is obtained as the limit of smooth solutions to a sequence of fully nonlinear Loewner-Nirenberg problems on approximating cones containing $(1,0,\dots,0)$, for which we also have uniqueness. In the process, we obtain an existence and uniqueness result for the corresponding Dirichlet boundary value problem with finite boundary data, which is also of independent interest. An important feature of our paper is that the existence of a conformal metric $g$ satisfying $\lambda(-g^{-1}A_g)\in\Gamma$ on $M$ is a \textit{consequence} of our results, rather than an assumption. 
	\end{abstract}

	\tableofcontents

	\section{Introduction}
	
	A pertinent theme in conformal geometry is to establish the existence of conformal metrics satisfying some notion of constant curvature. For example, given a compact Riemannian manifold $(M^n,g_0)$ of dimension $n\geq3$ with non-empty boundary $\partial M$, a natural question is whether there exists a conformal metric which is complete on $M\backslash \partial M$ and has constant negative scalar curvature on $M\backslash \partial M$. In the seminal work of Loewner \& Nirenberg \cite{LN74}, the authors proved among other results the existence and uniqueness of such a metric when $M\backslash \partial M$ is a bounded Euclidean domain with smooth boundary\footnote{Loewner \& Nirenberg also considered in \cite{LN74} the problem on a class of non-smooth Euclidean domains, but we will not be concerned with such generalisations in this paper.} and $g_0$ is the flat metric. Aviles \& McOwen \cite{AM88} later extended this result to the Riemannian setting, and for some further related results we refer e.g.~to the work of Allen et.~al.~\cite{AILA18}, Andersson et.~al.~\cite{ACF92}, Aviles \cite{Av82}, Finn \cite{Finn98}, Gover \& Waldron \cite{GW17}, Graham \cite{Gr17}, Han et.~al.~\cite{HJS20}, Han \& Shen \cite{HS20}, Jiang \cite{Jia21}, Li \cite{Li22}, Mazzeo \cite{Maz91} and V\'eron \cite{Ver81}. We note that the related problem of finding conformal metrics with constant scalar curvature on closed manifolds, known as the Yamabe problem, was solved by the combined works of Yamabe \cite{Yam60}, Trudinger \cite{Tru68}, Aubin \cite{Aub70} and Schoen \cite{Sch84}.

	Since the work of Viaclovsky \cite{Via00a} and Chang, Gursky \& Yang \cite{CGY02a}, there has been significant interest in fully nonlinear generalisations of Yamabe-type problems, including on manifolds with boundary. Suppose that 
	\begin{align}
	& \Gamma\subset\mathbb{R}^n\text{ is an open, convex, connected symmetric cone with vertex at 0}, \label{21'} \\
	& \Gamma_n^+ = \{\lambda\in\mathbb{R}^n: \lambda_i > 0 ~\forall ~1\leq i \leq n\} \subseteq \Gamma \subseteq \Gamma_1^+ =  \{\lambda\in\mathbb{R}^n : \lambda_1+\dots+\lambda_n > 0\}, \label{22'} \\
	& f\in C^\infty(\Gamma)\cap C^0(\overline{\Gamma}) \text{ is concave, 1-homogeneous and symmetric in the }\lambda_i, \label{23'}  \\
	& f>0 \text{ in }\Gamma, \quad f = 0 \text{ on }\partial\Gamma, \quad f_{\lambda_i} >0 \text{ in } \Gamma \text{ for }1 \leq i \leq n. \label{24'}
	\end{align}
	In this paper, we study the natural generalisation of the Loewner-Nirenberg problem to the fully nonlinear setting on Riemannian manifolds. That is, for $(f,\Gamma)$ satisfying \eqref{21'}--\eqref{24'} and a compact Riemannian manifold $(M,g_0)$ with non-empty boundary $\partial M$, we study the existence and uniqueness of a conformal metric $g_u = e^{2u}g_0$ satisfying
	\begin{equation}\label{113}
	\begin{cases} f(\lambda(-g_u^{-1}A_{g_u})) = 1, \quad \lambda(-g_u^{-1}A_{g_u}) \in \Gamma & \mathrm{on~}M\backslash \partial M \\
	u(x)\rightarrow+\infty & \mathrm{as~}\operatorname{d}(x,\partial M)\rightarrow 0.
	\end{cases} 
	\end{equation}
	Here, 
	\begin{align*}
A_g = \frac{1}{n-2}\bigg(\operatorname{Ric}_g - \frac{R_g}{2(n-1)}g\bigg)
\end{align*}
	denotes the $(0,2)$-Schouten tensor of a Riemannian metric $g$, $\operatorname{Ric}_g$ and $R_g$ respectively denote the Ricci curvature tensor and scalar curvature of $g$, $\lambda(T)$ denotes the vector of eigenvalues of a $(1,1)$-tensor $T$, and $\operatorname{d}(x,\partial M)$ is the distance from $x\in M$ to $\partial M$ with respect to $g_0$. Typical examples of $(f,\Gamma)$ satisfying \eqref{21'}--\eqref{24'} are given by $(\sigma_k^{1/k},\Gamma_k^+)$ for $1\leq k \leq n$, where $\sigma_k$ is the $k$'th elementary symmetric polynomial and $\Gamma_k^+ = \{\lambda\in\mathbb{R}^n:\sigma_j(\lambda)>0~\forall~1\leq j \leq k\}$. When $f=\sigma_1$, \eqref{113} reduces to the original Loewner-Nirenberg problem on Riemannian manifolds discussed above.

	 Much of the motivation to study \eqref{113} stems from the fact that, as a consequence of the Ricci decomposition, the Schouten tensor fully determines the conformal transformation properties of the full Riemann curvature tensor. We note that for $g_u = e^{2u}g_0$, one has the conformal transformation law 
	 \begin{equation}\label{118}
	 A_{g_u} = -\nabla_{g_0}^2 u  - \frac{1}{2}|\nabla_{g_0} u|_{g_0}^2 g_0 + du\otimes du + A_{g_0},
	 \end{equation}
	 which demonstrates the fully nonlinear nature of \eqref{113} when $f\not=c\sigma_1$. Moreover, \eqref{113} is non-uniformly elliptic when $f\not=c\sigma_1$.

By the 1-homogeneity of $f$, without loss of generality we may assume 
	\begin{align}\label{25'}
	f\bigg(\frac{1}{2},\dots, \frac{1}{2}\bigg) = 1. 
	\end{align}
	As in \cite{LN14b}, we define $\mu_\Gamma^+$ to be the number satisfying
	\begin{align*}
	(-\mu_\Gamma^+,1,\dots,1)\in\partial\Gamma.
	\end{align*}
	We note that $\mu_\Gamma^+$ is uniquely determined by $\Gamma$ and is easily seen to satisfy $\mu_\Gamma^+\in[0,n-1]$. When $\Gamma = \Gamma_k^+$, one has $\mu_{\Gamma_k^+}^+ = \frac{n-k}{k}$.
	
	Our first main result concerns the solution to the Loewner-Nirenberg problem \eqref{113} under the assumption 
	\begin{align}\label{418}
	\mu_\Gamma^+>1.
	\end{align}
Observe that for $\Gamma = \Gamma_k^+$,  \eqref{418} holds if and only if $k<\frac{n}{2}$. The role of condition \eqref{418} will be discussed later in the introduction.

	\begin{thm}\label{A}
		Let $(M,g_0)$ be a smooth compact Riemannian manifold of dimension $n\geq 3$ with non-empty boundary $\partial M$, and suppose $(f,\Gamma)$ satisfies \eqref{21'}--\eqref{24'}, \eqref{25'} and \eqref{418}. Then there exists a locally Lipschitz viscosity solution to \eqref{113} satisfying
		\begin{align}\label{114}
		\lim_{\operatorname{d}(x,\partial M)\rightarrow 0}\big(u(x) + \ln \operatorname{d}(x,\partial M)\big) = 0,
		\end{align}
		which is maximal in the sense that if $\widetilde{u}$ is any continuous viscosity solution to \eqref{113}, then $\widetilde{u} \leq u$ on $M\backslash \partial M$. Moreover, when $(1,0,\dots,0)\in\Gamma$, the solution $u$ is smooth and is the unique continuous viscosity solution to \eqref{113}. 
	\end{thm}

We recall that a continuous function $u$ on $M\backslash \partial M$ is a viscosity subsolution (resp.~viscosity supersolution) to the equation in \eqref{113} if for any $x_0\in M\backslash \partial M$ and $\phi\in C^2(M\backslash\partial M)$ satisfying $u(x_0) = \phi(x_0)$ and $u(x)\leq \phi(x)$ near $x_0$ (resp.~$u(x)\geq\phi(x)$ near $x_0$), it holds that $\lambda(-g_\phi^{-1} A_{g_{\phi}})(x_0)\in\{\lambda\in\Gamma:f(\lambda)\geq 1\}$ (resp.~$\lambda(-g_\phi^{-1} A_{g_{\phi}})(x_0)\in\mathbb{R}^n\backslash\{\lambda\in\Gamma:f(\lambda)> 1\}$). We say that $u$ is a viscosity solution to the equation in \eqref{113} if it is both a viscosity subsolution and a viscosity supersolution.

\begin{rmk}
	In previous work studying equations of the form $f(\lambda(-g_u^{-1}A_{g_u}))=1$, it has been typical to assume that the background metric $g_0$ satisfies $\lambda(-g_0^{-1}A_{g_0})\in\Gamma$ on $M$ (a notable exception is a result of Gursky, Streets \& Warren \cite{GSW11}, which will be discussed later in the introduction). In contrast, one of the key points of this paper is that we do not assume the existence of such a metric in Theorem \ref{A}. Rather, the existence of such a metric is established at an intermediate stage of the proof of Theorem \ref{A} (see Theorem \ref{55}).
\end{rmk}

\begin{rmk}
	In the case that $M\backslash \partial M$ is a Euclidean domain, the existence of a Lipschitz viscosity solution to \eqref{113} was established by Gonz\'ales, Li \& Nguyen in \cite{GLN18}. It was also shown in \cite{GLN18} that this solution is unique among continuous viscosity solutions. We note that the uniqueness of the viscosity solution obtained in Theorem \ref{A} remains an open problem when $M\backslash \partial M$ is not a Euclidean domain and $(1,0,\dots,0)\in\partial\Gamma$.
\end{rmk}

\begin{rmk}
	In \cite{LN20b, LNX22} it was shown that if $M\backslash \partial M$ is a Euclidean domain with disconnected boundary and $\Gamma\subset\Gamma_2^+$ (in particular, this implies $(1,0,\dots,0)\in\partial\Gamma$), then the Lipschitz viscosity solution to \eqref{113} is not differentiable. Thus, in general, the Lipschitz regularity of the solution in Theorem \ref{A} cannot be improved to $C^1$ regularity when $(1,0,\dots,0)\in\partial \Gamma$. On the other hand, the smoothness of the solution in Theorem \ref{A} when $(1,0,\dots,0)\in\Gamma$ is new even when $M\backslash \partial M$ is a Euclidean domain. This smoothness result can be viewed as an analogue of the result of Gursky \& Viaclovsky \cite{GV03b} on the existence of a smooth solution to the $\sigma_k$-Yamabe problem for the trace-modified Schouten tensor on closed manifolds. 
\end{rmk}

To describe the proof of Theorem \ref{A}, we first introduce some notation and an equivalent formulation of the result. For $\tau\in[0,1]$, $\lambda\in\mathbb{R}^n$ and $e=(1,\dots,1)\in\mathbb{R}^n$, we define 
\begin{equation*}
\lambda^\tau \defeq \tau\lambda + (1-\tau)\sigma_1(\lambda)e, \quad f^\tau(\lambda) \defeq \frac{1}{\tau+n(1-\tau)} f(\lambda^\tau) \quad \text{and} \quad \Gamma^\tau \defeq \{\lambda:\lambda^\tau\in\Gamma\}. 
\end{equation*}
As shown in \cite[Appendix A]{DN22}, $\Gamma$ satisfies \eqref{21'}, \eqref{22'} and $(1,0,\dots,0)\in\Gamma$ if and only if there exists $\widetilde{\Gamma}$ satisfying \eqref{21'}, \eqref{22'} and a number $\tau<1$ for which $\Gamma = (\widetilde{\Gamma})^\tau$. Note that \eqref{25'} implies $f^\tau(\frac{1}{2},\dots,\frac{1}{2})=1$. An equivalent formulation of Theorem \ref{A} is then as follows:

\begin{customthm}{\ref{A}$'$}\label{A'}
	\textit{Let $(M,g_0)$ be a smooth compact Riemannian manifold of dimension $n\geq 3$ with non-empty boundary $\partial M$, and suppose $(f,\Gamma)$ satisfies \eqref{21'}--\eqref{24'}, \eqref{25'} and \eqref{418}. Then for each $\tau<1$, there exists a smooth solution $u$ to
		\begin{equation}\label{-113}
		\begin{cases} f^\tau(\lambda(-g_u^{-1}A_{g_u})) = 1, \quad \lambda(-g_u^{-1}A_{g_u}) \in \Gamma^\tau & \mathrm{on~}M\backslash\partial M \\
		u(x)\rightarrow+\infty & \mathrm{as~}\operatorname{d}(x,\partial M)\rightarrow 0,
		\end{cases} 
		\end{equation}
		and moreover $u$ satisfies \eqref{114} and is the unique continuous viscosity solution to \eqref{-113}. When $\tau=1$, there exists a Lipschitz viscosity solution $u$ to \eqref{-113} satisfying \eqref{114}, which is maximal in the sense that if $\widetilde{u}$ is any continuous viscosity solution to \eqref{-113}, then $\widetilde{u} \leq u$ on $M\backslash \partial M$.}
\end{customthm}

\begin{rmk}
	If we label the solution to \eqref{-113} in Theorem \ref{A'} as $u^\tau$ for each $\tau\leq 1$, then we will show that for each compact set $K\subset M\backslash \partial M$, there exists a constant $C$ which is independent of $\tau$ but dependent on $M,g_0, f,\Gamma$ and $K$ such that 
	\begin{align*}
	\|u^\tau\|_{C^{0,1}(K)} \leq C \quad \text{for all } \tau\in[0,1].
	\end{align*}
\end{rmk}

In the proof of Theorem \ref{A'}, we will first prove the existence of a unique smooth solution to \eqref{-113} when $\tau<1$. The Lipschitz viscosity solution in the case $\tau=1$ is then obtained in the limit as $\tau\rightarrow 1$. In turn, for each $\tau<1$, the existence of a smooth solution to \eqref{-113} is obtained as the limit of smooth solutions to Dirichlet boundary value problems with finite boundary data. Although we only need to consider constant boundary data in the proof of Theorem \ref{A'}, we will prove the following more general result: 

\begin{thm}\label{55}
	Let $(M,g_0)$ be a smooth compact Riemannian manifold of dimension $n\geq 3$ with non-empty boundary $\partial M$, and suppose $(f,\Gamma)$ satisfies \eqref{21'}--\eqref{24'} and \eqref{418}. Let $\psi\in C^\infty(M)$ be positive and $\xi\in C^\infty(\partial M)$. Then for each $\tau<1$, there exists a smooth solution $u$ to 
	\begin{equation}\label{12}
	\begin{cases}
	f^\tau(\lambda(-g_u^{-1}A_{g_u})) = \psi, \quad \lambda(-g_u^{-1}A_{g_u})\in\Gamma^\tau & \mathrm{on~}M\backslash\partial M \\
	u = \xi & \mathrm{on~}\partial M,
	\end{cases}
	\end{equation}
	and moreover $u$ is the unique continuous viscosity solution to \eqref{12}. When $\tau=1$, there exists a Lipschitz viscosity solution to \eqref{12}.
\end{thm}
\begin{rmk}
	If we label the solution to \eqref{12} in Theorem \ref{55} as $u^\tau$ for each $\tau\leq1$, then we will show that there exists a constant $C$ which is independent of $\tau$ but dependent on $M,g_0, f,\Gamma, \psi$ and $\xi$ such that
	\begin{align*}
	\|u^\tau\|_{C^{0,1}(M)}\leq C \quad \text{for all }\tau\in[0,1].
	\end{align*}
\end{rmk}

The existence of a smooth solution to \eqref{12} when $\tau<1$ is achieved using the continuity method, which relies on obtaining \textit{a priori} estimates. To keep the introduction concise, we only discuss the $C^0$ estimates here and postpone the discussion of the other estimates to the main body of the paper. Now, if one assumes $\lambda(-g_0^{-1}A_{g_0})\in\Gamma$ on $M$, then it is straightforward to obtain both the \textit{a priori} upper and lower bounds on solutions to \eqref{12}. Since we do not make such an assumption on $g_0$, a large portion of our work involves proving the lower bound. The \textit{a priori} lower bound is obtained in two independent stages, which can be summarised as follows:
\begin{enumerate}
	\item First, in Section \ref{14}, we prove a local interior gradient estimate on solutions to \eqref{12} of the form 
	\begin{equation}\label{69'}
	|\nabla_{g_0} u|_{g_0}(x) \leq C(r^{-1} + e^{\sup_{B_r}u}) \quad\text{for }x\in B_{r/2},
	\end{equation}
	where $B_r$ is a geodesic ball contained in the interior of $M$. An important feature is that the estimate \eqref{69'} does not depend on a lower bound for $u$. 
	
	\item Second, in Section \ref{147}, we construct suitable barrier functions to prove a lower bound for $u$ in a \textit{uniform neighbourhood of $\partial M$} -- this is one of the key new ideas in this paper. 
\end{enumerate}
We note that the assumption $\mu_\Gamma^+>1$ is used in both stages above. Once the lower bound in a uniform neighbourhood of $\partial M$ is established in the second step, the local interior gradient estimate from the first step and a trivial global upper bound in Proposition \ref{75} then allows one to propagate the lower bound to all of $M$ -- see the proof of Proposition \ref{151} for the details.

We now discuss the two steps above in more detail. Our local interior gradient estimate, which is also of independent interest, is as follows: 
\begin{thm}\label{40}
	Let $(M,g_0)$ be a smooth Riemannian manifold of dimension $n\geq 3$, possibly with non-empty boundary, and suppose $(f,\Gamma)$ satisfies \eqref{21'}--\eqref{24'} and \eqref{418}. Fix $\tau\in(0,1]$, a positive function $\psi\in C^\infty(M)$ and suppose that $u\in C^3(B_r)$ satisfies
	\begin{equation}\label{10}
	f^\tau(\lambda(-g_u^{-1}A_{g_u})) = \psi, \quad \lambda(-g_u^{-1}A_{g_u})\in\Gamma^\tau 
	\end{equation}
	in a geodesic ball $B_r$ contained in the interior of $M$. Then 
	\begin{equation}\label{69}
	|\nabla_{g_0} u|_{g_0}(x) \leq C(r^{-1} + e^{\sup_{B_r}u}) \quad\text{for }x\in B_{r/2}
	\end{equation}
	where $C$ is a constant depending on $n, f, \Gamma$, $\|g_0\|_{C^3(B_r)}$ and $\|\psi\|_{C^1(B_r)}$ but independent of $\tau$ and $\inf_{B_r}\psi$. 
\end{thm}
We note that Theorem \ref{40} was previously obtained for $(f,\Gamma)=(\sigma_k^{1/k}, \Gamma_k^+)$ when $k<\frac{n}{2}$ and $\tau=1$ in the thesis of Khomrutai \cite{Kho09}\footnote{We would like to thank Baozhi Chu, YanYan Li and Zongyuan Li for bringing \cite{Kho09} to our attention.}. Roughly speaking, one important observation in \cite{Kho09} is as follows: if $\rho|\nabla_{g_0} u|_{g_0}^2$ attains its maximum at $x_0$ (here $\rho$ is a cutoff function satisfying standard assumptions), then in a `worst case scenario' (i.e.~in a situation where the gradient estimate cannot be obtained somewhat directly), the ordered eigenvalues $\lambda_1(x_0) \geq \dots \geq \lambda_n(x_0)$ of $(-g_0^{-1}A_{g_u})(x_0)$ are greater than or equal to a perturbation of $(1,\dots,1,-1)\frac{|\nabla u|^2}{2}(x_0)$. But when $k<\frac{n}{2}$, the vector $(1,\dots,1,-1)$ belongs to $\Gamma_k^+$, and so by the equation \eqref{10} and homogeneity of $\sigma_k^{1/k}$, the gradient estimate follows. In our proof of Theorem \ref{40}, we show that this phenomenon persists for general cones satisfying $\mu_\Gamma^+>1$. In order to circumvent certain arguments of Khomrutai that rely on algebraic properties of the $\sigma_k$ operators, we appeal to some general cone properties recently observed by Yuan in \cite{Yuan22}.

\begin{rmk}
	For gradient estimates on solutions to equations of the form \eqref{10} which depend on two-sided $C^0$ bounds, see for instance \cite{GV03b, Guan08}. For gradient estimates for the related positive cone equation, see e.g.~\cite{Che05, GW03b, JLL07, LL03, Li09, Wan06, Via02}. 
\end{rmk}

\begin{rmk}\label{179}
	We have been informed that in an upcoming work of Baozhi Chu, YanYan Li and Zongyuan Li \cite{CLL23}, a Liouville-type theorem for a fully nonlinear, degenerate elliptic Yamabe-type equation on negative cones is proved for all $\mu_\Gamma^+ \neq 1$. As an application of this Liouville-type theorem and the method in \cite{Li09} (which dealt with local gradient estimates for equations on positive cones), the authors obtain local interior gradient estimates for solutions to \eqref{10} depending	only on one-sided $C^0$ bounds for all $\mu_\Gamma^+ \neq 1$, without assuming concavity of $f$. Counterexamples to both results are also given when $\mu_\Gamma^+ = 1$. This proof is entirely different from our proof of Theorem \ref{40}.
\end{rmk}

We now turn to the second step mentioned above, namely the lower bound in a neighbourhood of $\partial M$. This is achieved through constructing suitable comparison functions on small annuli; the main step here is to prove the following proposition (see Proposition \ref{54} for a more precise version):
	\begin{prop}\label{54'}
	Suppose $(f,\Gamma)$ satisfies \eqref{21'}--\eqref{24'} and \eqref{418}, let $g_0$ be a Riemannian metric defined on a neighbourhood $\Omega$ of the origin in $\mathbb{R}^n$, and let $m\in\mathbb{R}$. Then there exist constants $S>1$ and $0<R<1$ depending on $g_0, f, \Gamma$ and $m$ such that whenever $1<\frac{r_+}{r_-}<S$ and $r_+<R$, there exists a solution to
	\begin{align*}
	\begin{cases}
	f(\lambda(-g_w^{-1}A_{g_w})) \geq 1, \quad \lambda(-g_w^{-1}A_{g_w})\in\Gamma &  \text{on }A_{r_-, r_+}\defeq \{x : r_-<\operatorname{d}_{g_0}(x,0)<r_+\} \\
	w(x) = m & \text{for }x\in\mathbb{S}_{r_-} \\
	w(x)\rightarrow -\infty & \text{as }\operatorname{d}_{g_0}(x,\mathbb{S}_{r_+})\rightarrow 0.
	\end{cases}
	\end{align*}
\end{prop}
Our construction of $w$ in Proposition \ref{54'} is modelled on the radial solutions of Chang, Han \& Yang \cite{CHY05} to the $\sigma_k$-Yamabe equation on annular domains in $\mathbb{R}^n$ when $k<\frac{n}{2}$. To apply Proposition \ref{54'} to complete the second step, we attach a collar neighbourhood $N$ to $\partial M$, and cover a neighbourhood of $\partial M$ in $M$ by sufficiently small annuli whose centres lie in $N$ and whose inner boundaries touch $\partial M$. On each of these annuli, the solutions constructed in Proposition \ref{54'} then serve as the desired lower bound by the comparison principle. We refer the reader to the proof of Proposition \ref{41} for the details. 

\begin{rmk}\label{182}
	The assumption $\mu_\Gamma^+>1$ plays an important role in our proof of Proposition \ref{54'}, and in fact a similar construction is not possible when $\mu_\Gamma^+\leq 1$. More precisely, given a smooth metric $g_0$ defined on an annulus $A_{r,R}$, and given a cone $\Gamma$ satisfying \eqref{21'}, \eqref{22'} and $\mu_\Gamma^+\leq 1$, there is no smooth metric $g_w = e^{2w}g_0$ satisfying $\lambda(-g_w^{-1}A_{g_w})\in\Gamma$ on $A_{r,R}$ and for which $w\rightarrow-\infty$ at either boundary component of $A_{r,R}$. The proof of this non-existence result uses arguments different in nature to those considered in this paper, and will appear elsewhere.
\end{rmk}

	For the remainder of the introduction, we discuss in more detail how our results and methods compare to previous work on fully nonlinear problems of Loewner-Nirenberg type. As mentioned before, when $M\backslash \partial M$ is a Euclidean domain, the existence of a Lipschitz viscosity solution to \eqref{113}, as well as uniqueness of this solution among continuous viscosity solutions, was established in \cite{GLN18}. Moreover, counterexamples to $C^1$ regularity were given in \cite{LN20b, LNX22}. The proof in \cite{GLN18} uses Perron's method, which in turn uses canonical solutions on interior/exterior balls and a comparison principle on Euclidean domains established in \cite{LNW18}. Since one cannot use exterior balls in the Riemannian setting, and since it is not currently known whether the comparison principle in \cite{LNW18} extends to the Riemannian setting, a different approach to that in \cite{GLN18} is required to prove Theorem \ref{A'}.
	
	On the other hand, for $(f,\Gamma) = (\sigma_k^{1/k},\Gamma_k^+)$ ($2\leq k \leq n$), Gursky, Streets \& Warren proved in \cite{GSW11} the existence of a unique smooth solution to \eqref{113} with the Ricci tensor in place of the Schouten tensor (see Remark \ref{168} below for the relation between this result and Theorem \ref{A'}, and see also the work of Wang \cite{W21} and Li \cite{L22} for some further related results). As in the present paper, the solution of Gursky, Streets \& Warren is constructed as a limit of solutions with finite boundary data, and these solutions are in turn obtained using the continuity method. Their method for obtaining an \textit{a priori} lower bound on solutions is different to ours, and is instead based on the explicit construction of a global subsolution. Roughly speaking, the subsolution construction in \cite{GSW11} uses the fact that, in the analogous formula to \eqref{118} for the Ricci tensor, the gradient terms are collectively nonnegative definite and so can be neglected in certain computations. In our case, the gradient terms do not have an overall sign, thus leading to our new approach for the lower bound discussed above.

	\begin{rmk}\label{168}
		Since $\mu_{\Gamma_k^+}^+ = \frac{n-k}{k}$, it is easy to see that $\mu_{(\Gamma_k^+)^\tau}^+ = \frac{n-k}{k} + (n-1)(1-\tau)$. Thus $\mu_{(\Gamma_k^+)^\tau}^+ >1$ if and only if $\tau < a_{n,k} \defeq \frac{n-k+k(n-2)}{k(n-1)}$. On the other hand, for $\tau=\frac{n-2}{n-1}$ we have $(\sigma_k^{1/k})^\tau(\lambda(-g_u^{-1}A_{g_u})) = \frac{1}{n-1}\sigma_k^{1/k}(\lambda(-g_u^{-1}\operatorname{Ric}_{g_u}))$. Since $\frac{n-2}{n-1}<a_{n,k}$ if and only if $k<n$, we therefore see that Theorem \ref{A'} recovers the result of \cite{GSW11} for $k<n$. 
	\end{rmk}

The plan of the paper is as follows. In Section \ref{14} we prove the local interior gradient estimate stated in Theorem \ref{40}. In Section \ref{15} we consider the Dirichlet boundary value problem \eqref{12}, proving Theorem \ref{55}. Finally, in Section \ref{135} we turn to the fully nonlinear Loewner-Nirenberg problem \eqref{-113}, proving Theorem \ref{A'} (and hence Theorem \ref{A}). \medskip
	
\noindent \textbf{Notation:} Throughout the rest of the paper, if $X$ is a $(1,1)$-tensor satisfying $\lambda(X)\in\Gamma$ then we frequently denote $f(X)\defeq f(\lambda(X))$. \medskip 

\noindent \textbf{Acknowledgements:} The authors would like to thank Prof.~YanYan Li for stimulating discussions and his constant support.

\section{Proof of Theorem \ref{40}: the local interior gradient estimate}\label{14}

In this section we prove the local interior gradient estimate stated in Theorem \ref{40}. Throughout the section, unless otherwise stated all derivatives and norms are taken with respect to $g_0$. Moreover, $C$ will denote a constant that may change from line to line and depends only on $n, f, \Gamma, \|g_0\|_{C^3(B_r)}$ and $\|\psi\|_{C^1(B_r)}$.

\subsection{Set-up and main ideas of the proof}\label{302}
	
Our set-up for the proof of Theorem \ref{40} is similar to that in the related works \cite{Che05, GW03b, JLL07, LL03, Kho09, Li09, Wan06} on local gradient estimates. Throughout this section we denote $S=A_{g_0}$ and

\begin{equation*}
	W = \nabla^2 u  + \frac{1}{2}|\nabla u|^2 g_0 - du\otimes du - S.
\end{equation*}
By a standard argument it suffices to consider the case $r=1$ in the proof of Theorem \ref{40}. Suppose $\rho\in C_c^\infty(B_1)$ is a cutoff function in $B_1$ with $\rho = 1$ on $B_{1/2}$, $|\nabla \rho| \leq C \rho^{1/2}$ and $|\nabla^2 \rho|\leq C$. Set $H=\rho|\nabla u|^2$ and suppose $H$ attains a maximum at $x_0$. We may assume that $|\nabla u|\geq 1$ at $x_0$, otherwise we are done. Choosing suitable normal coordinates centred at $x_0$, we may also assume $W = (w_{ij})$ is diagonal at $x_0$ with $w_{11} \geq \dots \geq w_{nn}$, and hence at $x_0$ we have
	\begin{align}\label{21}
	\begin{cases}w_{ii} = u_{ii}  - u_i^2 + \frac{1}{2}|\nabla u|^2 - S_{ii} & \text{for all }1\leq i \leq n, \\
	u_{ij}  = u_i u_j + S_{ij} &\text{for }i\not=j. 
	\end{cases} 
	\end{align}
	Using the fact that $H_i(x_0)=0$ for each $i$, we obtain at $x_0$ 
	\begin{equation}\label{120}
	\sum_{l=1}^n u_{il}u_l = -\frac{\rho_i}{2\rho}|\nabla u|^2
	\end{equation}
	and hence 
	\begin{equation}\label{4}
	\bigg|\sum_{l=1}^n u_{il}u_l \bigg| \leq C\rho^{-1/2}|\nabla u|^2.
	\end{equation}
	For $A_0$ a large number to be fixed later, we may assume at $x_0$ that
	\begin{equation}\label{5}
	\rho^{-1/2} \leq C\frac{|\nabla u|}{A_0}\quad\text{and}\quad|S|\leq \frac{|\nabla u|^2}{A_0},
	\end{equation}
	otherwise we are done. Note that by combining \eqref{4} with the first estimate in \eqref{5}, we have
	\begin{align}\label{26}
	\bigg|\sum_{l=1}^n u_{il}u_l \bigg| \leq C\frac{|\nabla u|^3}{A_0}. 
	\end{align}

	Denote by $F_\tau^{ij}$ the coefficients of the linearised operator at $ (g_0^{-1}W)(x_0)$, that is
	\begin{align*}
	F_\tau^{ij}= \frac{\partial f^\tau}{\partial A_{ij}}\bigg|_{A=(g_0^{-1}W)(x_0)}.
	\end{align*}
	Then $(F^{ij}_\tau)$ is a positive definite, diagonal matrix. Also denote
	\begin{align*}
	\mathcal{F}_\tau = \sum_{i=1}^n F^{ii}_\tau \quad \text{and} \quad \widetilde{u}_{ij} \defeq u_{ij} - S_{ij}. 
	\end{align*}
	By homogeneity and concavity of $f$, it is easy to see that $\mathcal{F}_\tau \geq \frac{1}{C}>0$: indeed, denoting $\lambda = \lambda(g_0^{-1}W)(x_0)$, we have
	\begin{align}\label{183}
	\mathcal{F}_\tau = \sum_{i=1}^n \frac{\partial f^\tau}{\partial \lambda_i}(\lambda) = f^\tau(\lambda) + \sum_{i=1}^n \frac{\partial f^\tau}{\partial \lambda_i}(\lambda)(1-\lambda_i) \geq f^\tau(1,\dots,1). 
	\end{align}
	
	With our set-up and notation established, we now briefly discuss the main ideas in the proof of Theorem \ref{40}. The first step is to obtain the following lemma: 
	
	\begin{lem}\label{34}
		Under the same hypotheses as Theorem \ref{40} but without the restriction $\mu_\Gamma^+>1$, there exists a constant $C$ such that 
		\begin{align}\label{1''}
		0 \geq- C\mathcal{F}_\tau(1+e^{2u})|\nabla u|^2 - C\rho\mathcal{F}_\tau\frac{|\nabla u|^4}{A_0}  + \rho\sum_{i,l}F_\tau^{ii}\widetilde{u}_{il}^2 \quad \text{at }x_0. 
		\end{align}
	\end{lem}
\noindent The proof of Lemma \ref{34} is by now standard and will be given in Section \ref{132}.
	
	 Now, in the case that the positive term on the RHS of \eqref{1''} dominates $|\nabla u|^4\mathcal{F}_\tau$, in the sense that 
	\begin{align}\label{300}
	\sum_{i,l}F^{ii}_\tau \widetilde{u}_{il}^2 \geq \ep|\nabla u|^4\mathcal{F}_\tau \quad \text{at }x_0
	\end{align}
	for a suitably chosen small constant $\ep>0$, then the desired gradient estimate is routine (the details will be given later). On the other hand, if \eqref{300} fails for our suitably chosen small constant $\ep>0$, we will see that the ordered eigenvalues $w_{11} \geq \dots \geq w_{nn}$ of $W$ at $x_0$ are greater than or equal to a perturbation of $(1,\dots,1,-1)\frac{|\nabla u|^2}{2}$. As mentioned in the introduction, this phenomenon was previously observed in the case $(f,\Gamma) = (\sigma_k^{1/k},\Gamma_k^+)$ when $k<\frac{n}{2}$ in the thesis of Khomrutai \cite{Kho09}. Using the fact that $(1,\dots,1,-1)\in\Gamma$ (this is the only place in the proof of Theorem \ref{40} where the assumption $\mu_\Gamma^+>1$ is used), the gradient estimate again follows. The details will be given in Section \ref{301}.

	\subsection{Proof of Lemma \ref{34}}\label{132}
	
In this section we give the proof of Lemma \ref{34}:

\begin{proof}
	We follow closely the proof of Guan \& Wang \cite{GW03b}. In what follows, all computations are implicitly carried out at $x_0$. First observe that by \eqref{120},
	\begin{align*}
	H_{ij} = \bigg(\rho_{ij} - \frac{2\rho_i\rho_j}{\rho}\bigg)|\nabla u|^2 + 2\rho\sum_{l=1}^n u_{lij}u_l + 2\rho\sum_{l=1}^n u_{il}u_{jl},
	\end{align*}
	and hence by positivity of $(F^{ij}_\tau)$ and non-positivity of $(H_{ij})$, 
	\begin{align}\label{9}
	0 \geq \sum_{i=1}^nF_\tau^{ii}H_{ii} & = \sum_{i=1}^nF_\tau^{ii}\bigg[\bigg(\rho_{ii} - \frac{2\rho_i^2}{\rho}\bigg)|\nabla u|^2 + 2\rho\sum_{l=1}^n u_{lii}u_l + 2\rho\sum_{l=1}^n u_{il}^2\bigg] \nonumber \\
	& = -C |\nabla u|^2 \mathcal{F}_\tau + 2\rho\sum_{i,l} F_\tau^{ii}u_{lii}u_l +2\rho\sum_{i,l}F_\tau^{ii} u_{il}^2.
	\end{align}
	Now, commuting derivatives yields
	\begin{align}\label{8}
	\sum_{i,l} F_\tau^{ii}u_{lii}u_l & \geq \sum_{i,l} F_\tau^{ii}u_{iil} u_l - C|\nabla u|^2 \mathcal{F}_\tau \nonumber \\
	& = \sum_{i,l} F_\tau^{ii}\bigg[(w_{ii})_l - \bigg(\frac{1}{2}|\nabla u|^2 - u_i^2\bigg)_l + (S_{ii})_l \bigg]u_l   - C|\nabla u|^2\mathcal{F}_\tau \nonumber \\
	& = \sum_{l=1}^n (\psi e^{2u})_l u_l - \mathcal{F}_\tau\sum_{k,l} u_{kl}u_k u_l + 2\sum_{i,l} F_\tau^{ii} u_{il}u_i u_l + \sum_{i,l} F_\tau^{ii}(S_{ii})_l u_l - C|\nabla u|^2 \mathcal{F}_\tau,
	\end{align} 
	where to reach the last line we have used the fact that $f^\tau$ is homogeneous of degree one to assert that $\sum_i F^{ii}_\tau (w_{ii})_l = (f^\tau(g_0^{-1}W))_l= (\psi e^{2u})_l$. Also, since $|\nabla u|\geq 1$, we can bound the penultimate term in \eqref{8} from below by $-C|\nabla u|^2\mathcal{F}_\tau$, and also observe that
	\begin{align}\label{121}
	\sum_{l=1}^n (\psi e^{2u})_l u_l &  = \sum_{l=1}^n e^{2u}\psi_l u_l + 2 e^{2u}\psi|\nabla u|^2  \geq -Ce^{2u}|\nabla u|^2.
	\end{align}
	Also, by \eqref{26} we have 
	\begin{equation}\label{6}
	- \mathcal{F}_\tau \sum_{k,l} u_{kl}u_k u_l \geq -C\frac{|\nabla u|^4}{A_0}\mathcal{F}_\tau,
	\end{equation}
	and likewise 
	\begin{equation}\label{7}
	2\sum_{i,l} F_\tau^{ii} u_{il}u_i u_l = 2\sum_i \bigg(F^{ii}_\tau u_i\sum_l u_{il}u_l\bigg)  \geq - 2\sum_i \bigg(\big|F^{ii}_\tau u_i\big|\Big|\sum_l u_{il}u_l\Big|\bigg)\geq  -C \frac{|\nabla u|^4}{A_0}\mathcal{F}_\tau.  
	\end{equation}
	Substituting \eqref{121}--\eqref{7} back into \eqref{8} and recalling $\mathcal{F}_\tau \geq \frac{1}{C}$, we get
	\begin{align*}
	\sum_{i,l} F_\tau^{ii}u_{lii}u_l & \geq -C(1+e^{2u})|\nabla u|^2 \mathcal{F}_\tau - C\mathcal{F}_\tau\frac{|\nabla u|^4}{A_0},
	\end{align*}
	and substituting this back into \eqref{9} we see
	\begin{align}\label{1}
	0 & \geq - C\mathcal{F}_\tau(1+e^{2u})|\nabla u|^2 - C\rho\mathcal{F}_\tau\frac{|\nabla u|^4}{A_0} + 2\rho \sum_{i,l}F_\tau^{ii}u_{il}^2.
	\end{align} 
	The desired estimate \eqref{1''} then follows from \eqref{1} and the following inequality, which is a consequence of the Cauchy-Schwarz inequality and the second inequality in \eqref{5}:
	\begin{align*}
	\sum_{i,l} F_\tau^{ii}u_{il}^2  \geq  \frac{1}{2}\sum_{i,l} F_\tau^{ii} \widetilde{u}_{il}^2 - \frac{1}{A_0}\mathcal{F}_\tau|\nabla u|^4.
	\end{align*}
\end{proof}

\subsection{Proof of Theorem \ref{40}}\label{301}

We begin this section by stating a central result in our argument, namely Proposition \ref{303}. The proof of Theorem \ref{40} is then given assuming the validity of Proposition \ref{303} -- this should serve to elucidate the ideas outlined at the end of Section \ref{302}. The proof of Proposition \ref{303} will be given later in the section, and consists of a series of technical lemmas. 

To this end, for $1>\delta_0 \geq A_0^{-1/10}$ a small number to be fixed later, define the set
\begin{align*}
\mathcal{I} = \bigg\{ i\in \{1,\dots,n\}: \bigg|w_{jj} + \frac{|\nabla u|^2}{2}\bigg| < 2\delta_0^2 |\nabla u|^2\bigg\}. 
\end{align*}
We remind the reader that all computations are implicitly carried out at $x_0$, and that we have the ordering $w_{11} \geq \dots \geq w_{nn}$. We will prove: 

\begin{prop}\label{303}
	There exists a constant $\widetilde{C}>1$ depending only on $n, f, \Gamma, \|g_0\|_{C^3(B_r)}$ and $\|\psi\|_{C^1(B_r)}$ such that if $A_0^{-1/10} \leq \delta_0 \leq \widetilde{C}^{-1}$ and 
	\begin{align}\label{125}
	\sum_{i,l} F_\tau^{ii} \widetilde{u}_{il}^2 < \widetilde{C}^{-1}\delta_0^4|\nabla u|^4 \mathcal{F}_\tau, 
	\end{align}
	then: 
	\begin{enumerate}
		\item $\mathcal{I} = \{n\}$, and 
		\item $\big| w_{n-1,n-1} - \frac{|\nabla u|^2}{2}\big|<2\delta_0|\nabla u|^2$. 
	\end{enumerate}
\end{prop}

Assuming the validity of Proposition \ref{303} for now, let us complete the proof of Theorem \ref{40}: 

\begin{proof}[Proof of Theorem \ref{40}]
	We start by fixing $\widetilde{C}$ sufficiently large so that Proposition \ref{303} applies. Then for $A_0> \widetilde{C}^{10}$ to be fixed later, if  $A_0^{-1/10} \leq \delta_0 \leq \widetilde{C}^{-1}$ and \eqref{125} is satisfied, then
	\begin{equation*}
	w_{n-1, n-1} = (1+a_{n-1})\frac{|\nabla u|^2}{2} \quad \text{and} \quad w_{nn} =-(1+a_n)\frac{|\nabla u|^2}{2}
	\end{equation*}
	for some $|a_{n-1}|, |a_n| \leq 4\delta_0$. On the other hand, since $w_{11}\geq \dots \geq w_{nn}$, for each $\alpha = 1,\dots,n-2$ we can write $w_{\alpha\alpha} = w_{n-1,n-1} + X_\alpha$ for some $X_\alpha\geq 0$. Therefore
	\begin{align}\label{39}
	& \begin{pmatrix}
	w_{11} \\
	\vdots \\
	w_{n-2, n-2} \\
	w_{n-1, n-1} \\
	w_{nn}
	\end{pmatrix} = \begin{pmatrix}
	X_1 \\
	\vdots \\
	X_{n-2} \\
	0 \\
	0
	\end{pmatrix} + \frac{|\nabla u|^2}{2}\underbrace{\begin{pmatrix}
		1+a_{n-1} \\
		\vdots \\
		1+a_{n-1} \\
		1+a_{n-1} \\
		-(1+a_n)
		\end{pmatrix}}_{\mathcal{B}},
	\end{align}
	with the first vector on the RHS of \eqref{39} clearly belonging to $\overline{\Gamma^\tau}$ for each $\tau \leq 1$, since each entry is nonnegative. We also observe that $\mathcal{B}$ is a perturbation of $\mathcal{B}_0 \defeq (1,\dots,1,-1)$, and that $\mathcal{B}_0\in\Gamma^\tau$ for any $\tau\leq 1$ since we assume $\mu_\Gamma^+>1$. Therefore, since $|a_{n-1}|, |a_n| \leq 4\delta_0$, for $\widetilde{C}$ sufficiently large it will hold that $\mathcal{B}\in\Gamma$ with $f^\tau(\mathcal{B}) \geq \frac{1}{2}f^\tau(\mathcal{B}_0)$. Monotonicity of $f$ then implies 
	\begin{align*}
	\psi e^{2u} = f^\tau(w_{11},\dots,w_{nn}) \geq \frac{|\nabla u|^2}{2}f^\tau(\mathcal{B}) \geq  \frac{|\nabla u|^2}{4}f^\tau(\mathcal{B}_0),
	\end{align*}
	which implies the desired gradient estimate. 
	
	It remains to address the case that, for the value of $\widetilde{C}$ fixed in the foregoing argument, \eqref{125} is not satisfied. Then 
		\begin{equation}\label{2}
	\sum_{i,l} F_\tau^{ii} \widetilde{u}_{il}^2 \geq  \widetilde{C}^{-1}A_0^{-2/5}|\nabla u|^4 \mathcal{F}_\tau,
	\end{equation}
	and substituting \eqref{2} into \eqref{1''} we therefore have
	\begin{align*}
	0 \geq- C\mathcal{F}_\tau(1+e^{2u})|\nabla u|^2 - C\rho\mathcal{F}_\tau\frac{|\nabla u|^4}{A_0}  + \widetilde{C}^{-1}A_0^{-2/5}\rho|\nabla u|^4 \mathcal{F}_\tau.
	\end{align*}
	Multiplying through by $\widetilde{C}A_0^{2/5}\rho$ then yields the estimate
	\begin{align}\label{419}
	0 & \geq - \widetilde{C}CA_0^{2/5}\rho (1+e^{2u})|\nabla u|^2 - \frac{\widetilde{C}C}{A_0^{3/5}} \rho^2|\nabla u|^4 + \rho^2 |\nabla u|^4.
	\end{align}
	It follows that if we choose $A_0 \geq \max\{(2\widetilde{C}C)^{5/3}, \widetilde{C}^{10}\}$ (where $C$ and $\widetilde{C}$ are the constants in \eqref{419}), then we have (for a possibly different constant $C$)
	\begin{align}
	0 & \geq -C\rho(1+e^{2u})|\nabla u|^2 + \frac{1}{2}\rho^2 |\nabla u|^4,
	\end{align}
	and therefore
	\begin{equation}
	H^2 = \rho^2 |\nabla u|^4 \leq C(1+e^{2u}) H.
	\end{equation}
	After dividing through by $H$ we again arrive at the desired gradient estimate.
\end{proof}

The rest of the section is devoted to the proof of Proposition \ref{303}, which we obtain through a series of three lemmas. In the first of these lemmas we show that if $A_0^{-1/10} \leq \delta_0 \leq \widetilde{C}^{-1}$ for $\widetilde{C}$ sufficiently large, then $\mathcal{I}\not=\emptyset$:

\begin{lem}\label{68}
	There exists a constant $\widetilde{C} >1$ depending only on $n, f, \Gamma, \|g_0\|_{C^3(B_r)}$ and $\|\psi\|_{C^1(B_r)}$ such that if $A_0^{-1/10} \leq \delta_0 \leq \widetilde{C}^{-1}$, then $\mathcal{I}\not=\emptyset$. 
\end{lem}
\begin{proof}
	It is clear that for $\delta_0\leq \sqrt{1/n}$, there is at least one index $j\in \{1,\dots,n\}$ such that $u_j^2 \geq \delta_0^2|\nabla u|^2$. We claim that for such an index $j$, it holds that $j\in\mathcal{I}$. We follow the method of Guan \& Wang \cite{GW03b}. We know that for $l\not=j$, $u_{jl} = u_j u_l + S_{jl}$ and therefore
	\begin{equation*}
	\sum_{l\not=j} u_{jl}u_l = \sum_{l\not=j}u_ju_l^2 + \sum_{l\not=j} S_{jl}u_l. 
	\end{equation*}
	It follows that 
	\begin{align*}
	\sum_{l=1}^n u_{jl}u_l & = \sum_{l\not=j} u_ju_l^2 + \sum_{l\not=j} S_{jl}u_l + u_{jj}u_j \nonumber \\
	& = u_j|\nabla u|^2 + \sum_{l\not=j} S_{jl}u_l + u_{jj}u_j - u_j^3 \nonumber \\
	& = \sum_{l\not=j} S_{jl}u_l - u_j\bigg(\big(u_j^2 - |\nabla u|^2\big)- u_{jj}\bigg). 
	\end{align*}
	Hence
	\begin{align*}
	\bigg|u_j\bigg(\big(u_j^2 - |\nabla u|^2\big)- u_{jj}\bigg) - \sum_{l\not=j}S_{jl}u_l\bigg| = \bigg|\sum_{l=1}^n u_{jl}u_l\bigg| \stackrel{\eqref{26}}{\leq} C\frac{|\nabla u|^3}{A_0}. 
	\end{align*}
	It follows that
	\begin{align}\label{30}
	\big|u_j\big|\big|\big(u_j^2 - |\nabla u|^2\big)- u_{jj}\big| \leq   C\frac{|\nabla u|^3}{A_0} + \bigg|\sum_{l\not=j}S_{jl}u_l\bigg|  \stackrel{\eqref{5}}{\leq}  C\frac{|\nabla u|^3}{A_0} \leq C\delta_0^{10}|\nabla u|^3,
	\end{align}
	where to reach the last inequality we have used $A_0^{-1/10} \leq \delta_0$. Substituting $|u_j| \geq \delta_0|\nabla u|$ back into \eqref{30} yields
	\begin{align}\label{31}
	\big|\big(u_j^2 - |\nabla u|^2\big)- u_{jj}\big| \leq C\delta_0^9|\nabla u|^2.
	\end{align}
	Next, substituting $u_{jj} = w_{jj} + u_j^2 - \frac{1}{2}|\nabla u|^2 + S_{jj}$ into \eqref{31} and again applying \eqref{5} we obtain 
	\begin{align}\label{128}
	\bigg|w_{jj} + \frac{1}{2}|\nabla u|^2\bigg| \leq C\delta_0^9|\nabla u|^2 + \frac{|\nabla u|^2}{A_0} = C\delta_0^9|\nabla u|^2 + \delta_0^{10}|\nabla u|^2. 
	\end{align}
	It is clear that one can then choose $\widetilde{C}$ sufficiently large so that the right hand side of \eqref{128} is less than $2\delta_0^2|\nabla u|^2$ for $\delta_0 \leq \widetilde{C}^{-1}$. Once such a choice is made, we see that \eqref{128} implies $j\in\mathcal{I}$, which proves the claim and therefore the lemma.
\end{proof}

In our subsequent arguments we will use the following proposition, which is essentially a consequence of \cite[Theorem 1.4]{Yuan22} -- see Appendix \ref{AA} for a summary of the proof.

\begin{prop}\label{46}
	Suppose $\Gamma$ satisfies \eqref{21'} and \eqref{22'} with $\Gamma\not=\Gamma_n^+$ (equivalently, $\mu_\Gamma^+>0$). Then there exists a constant $\theta=\theta(n,\Gamma)>0$ such that for any $\lambda\in\Gamma$ with $\lambda_1\geq \dots \geq \lambda_n$, 
	\begin{equation}\label{37'}
	\frac{\partial f}{\partial \lambda_i}(\lambda) \geq \theta \sum_{j=1}^n \frac{\partial f}{\partial\lambda_j}(\lambda) \quad \text{if }i\in\{n-1,n\} \text{ or }\lambda_i\leq 0. 
	\end{equation}
\end{prop}

We are now in a position to show that if one additionally assumes \eqref{125} holds for $\widetilde{C}$ sufficiently large, then $|\mathcal{I}|=\{n\}$ (recall once again the ordering $w_{11} \geq \dots \geq w_{nn}$):

\begin{lem}\label{63}
	There exists a constant $\widetilde{C}>1$ depending only on $n, f, \Gamma, \|g_0\|_{C^3(B_r)}$ and $\|\psi\|_{C^1(B_r)}$ such that if $A_0^{-1/10} \leq \delta_0 \leq \widetilde{C}^{-1}$ and \eqref{125} is satisfied, then $|\mathcal{I}|=\{n\}$. 
\end{lem}
\begin{proof}
	We first claim that if $\widetilde{C}$ is sufficiently large and \eqref{125} holds, then $u_{jj}> -2\delta_0^2 |\nabla u|^2$ for $j\in \mathcal{I}$. Indeed, suppose for a contradiction that this is not the case. Then we would have
	\begin{align}\label{56}
	\sum_{i,l} F_\tau^{ii} \widetilde{u}_{il}^2  \geq F_\tau^{jj}\widetilde{u}_{jj}^2 \stackrel{\eqref{5}}{\geq} \frac{1}{2}F^{jj}_\tau u_{jj}^2 - F^{jj}_\tau\frac{|\nabla u|^4}{A_0^2} & \geq 2F_\tau^{jj}\delta_0^4|\nabla u|^4 - F^{jj}_\tau\delta_0^{20}|\nabla u|^4  \nonumber \\
	& \geq F^{jj}_\tau \delta_0^4|\nabla u|^4 \nonumber \\
	& \geq \theta \delta_0^4|\nabla u|^4 \mathcal{F}_\tau,
	\end{align}
	with the last inequality following from Proposition \ref{46} -- note that Proposition \ref{46} applies in this case, since $w_{jj}<0$ by virtue of $j\in\mathcal{I}$ if $\widetilde{C}$ is sufficiently large. But this contradicts \eqref{125} if $\widetilde{C}$ is sufficiently large, proving the claim.

	By the claim, we may therefore suppose that $\widetilde{C}$ is large enough so that $u_{jj}> -2\delta_0^2 |\nabla u|^2$ whenever $j\in \mathcal{I}$. Then for $j\in\mathcal{I}$, we therefore have
	\begin{align*}
	-2\delta_0^2|\nabla u|^2 - u_j^2 + \frac{1}{2}|\nabla u|^2 - S_{jj} < u_{jj} - u_j^2 + \frac{1}{2}|\nabla u|^2 - S_{jj} = w_{jj} < -\frac{1}{2}|\nabla u|^2 + 2\delta_0^2|\nabla u|^2,
	\end{align*}
	with the last inequality following from the definition of $\mathcal{I}$. That is,
	\begin{align}\label{57}
	-u_j^2 < (-1+4\delta_0^2)|\nabla u|^2 + S_{jj} \stackrel{\eqref{5}}{<} (-1+4\delta_0^2)|\nabla u|^2 + \delta_0^{10}|\nabla u|^2 < (-1+5\delta_0^2)|\nabla u|^2.
	\end{align}
	Clearly \eqref{57} cannot hold for more than one index if $10\delta_0^2<1$. Hence $|\mathcal{I}| \leq 1$ for $\widetilde{C}$ sufficiently large, and after increasing $\widetilde{C}$ further if necessary so that $\mathcal{I}\not=\emptyset$ (recall that this is possible by Lemma \ref{68}), then it must be the case that $|\mathcal{I}|=1$, i.e.~$\mathcal{I} = \{n\}$.
\end{proof}

To finish the proof of Proposition \ref{303} it remains to show (after taking $\widetilde{C}$ larger if necessary) that $| w_{n-1,n-1} - \frac{|\nabla u|^2}{2}|<2\delta_0|\nabla u|^2$. This is the focus of the next lemma:

\begin{lem}\label{65}
	There exists a constant $\widetilde{C}>1$ depending only on $n, f, \Gamma, \|g_0\|_{C^3(B_r)}$ and $\|\psi\|_{C^1(B_r)}$ such that if $A_0^{-1/10} \leq \delta_0 \leq \widetilde{C}^{-1}$ and \eqref{125} is satisfied, then 
	\begin{align*}
	\bigg| w_{n-1,n-1} - \frac{|\nabla u|^2}{2}\bigg|<2\delta_0|\nabla u|^2
	\end{align*}
\end{lem}
\begin{proof}
	\textbf{Step 1:} In this first step we show
	\begin{align}\label{58}
	w_{n-1,n-1} > \bigg(\frac{1}{2}-2\delta_0\bigg)|\nabla u|^2.
	\end{align}
Suppose for a contradiction that $w_{n-1,n-1} \leq (\frac{1}{2}-2\delta_0)|\nabla u|^2$, i.e.
	\begin{align}\label{59}
	u_{n-1,n-1} - u_{n-1}^2 - S_{n-1,n-1}  \leq  - 2\delta_0|\nabla u|^2. 
	\end{align}
	Either $u_{n-1}^2 < \delta_0|\nabla u|^2$ or $u_{n-1}^2 \geq \delta_0|\nabla u|^2$. In the former case, \eqref{59} then implies 
	\begin{align}
	u_{n-1,n-1} < -\delta_0|\nabla u|^2 + S_{n-1,n-1}  \stackrel{\eqref{5}}{<} -\delta_0|\nabla u|^2 + \delta_0^{10}|\nabla u|^2 < -\frac{1}{2}\delta_0|\nabla u|^2 \quad \text{if }\delta_0<\frac{1}{2},
	\end{align}
	and one obtains a contradiction as in \eqref{56} if $\widetilde{C}$ is sufficiently large -- note that Proposition \ref{46} is again justified, since $w_{n-1,n-1}$ is the second lowest eigenvalue. If instead $u_{n-1}^2 \geq \delta_0|\nabla u|^2$, the proof of Lemma \ref{68} shows that $n-1\in\mathcal{I}$. This contradicts the conclusion $|\mathcal{I}|=\{n\}$ of Lemma \ref{63} if $\widetilde{C}$ is sufficiently large. Thus \eqref{58} is established, which completes the proof of Step 1. \medskip 
	
	\noindent\textbf{Step 2:} In this second step we show 
	\begin{align}\label{64'}
	w_{n-1,n-1} < \bigg(\frac{1}{2}+2\delta_0\bigg)|\nabla u|^2.
	\end{align}
	Indeed, we have
	\begin{align*}
	w_{n-1,n-1} = u_{n-1,n-1} - u_{n-1}^2 + \frac{1}{2}|\nabla u|^2 - S_{n-1,n-1} \stackrel{\eqref{5}}{\leq} |u_{n-1,n-1}| + \frac{1}{2}|\nabla u|^2 + \delta_0^{10}|\nabla u|^2.
	\end{align*}
	But $|u_{n-1,n-1}| \leq \delta_0 |\nabla u|^2$, else one would obtain a contradiction as in \eqref{56} if $\widetilde{C}$ is sufficiently large (again we are using the fact $w_{n-1,n-1}$ is the second lowest eigenvalue, so Proposition \ref{46} applies). The estimate \eqref{64'} thus follows, which completes the proof of Step 2. \medskip
	
	With \eqref{58} and \eqref{64'} established, the proof of Lemma \ref{65} is complete. 
\end{proof}

\begin{proof}[Proof of Proposition \ref{303}]
	This is an immediate consequence of Lemmas \ref{68}, \ref{63} and \ref{65}. 
\end{proof}

\section{Proof of Theorem \ref{55}: the Dirichlet boundary value problem}\label{15}

As discussed in the introduction, in the proof of Theorem \ref{A'} we will first address the corresponding Dirichlet boundary value problem with finite boundary data. To this end, in this section we prove Theorem \ref{55}. Our proof uses the continuity method, and we proceed according to the following steps: 

\begin{enumerate}
	\item In Section \ref{146} we give a routine proof of the global upper bound on solutions for $\tau\leq 1$, independently of whether or not $\mu_\Gamma^+>1$. 
	\item In Section \ref{147} we prove the global lower bound on solutions for $\tau\leq 1$ when $\mu_\Gamma^+>1$. As outlined in the introduction, we use two main ingredients: our local interior gradient estimate obtained in Theorem \ref{40}, and a lower bound in a uniform neighbourhood of $\partial M$, which is obtained by constructing suitable comparison functions on small annuli (see Propositions \ref{41} and \ref{54}).
	\item In Section \ref{148} we prove the global gradient estimate for $\tau\leq 1$ when $\mu_\Gamma^+>1$. To obtain the lower bound for the normal derivative on $\partial M$ we use our comparison functions on small annuli constructed in Section \ref{147}, and to obtain the upper bound for the normal derivative on $\partial M$ we use comparison functions similar to that of Guan \cite{Guan08} (this latter argument does not use $\mu_\Gamma^+>1$). For the interior estimates we use Theorem \ref{40}, and for estimates near $\partial M$ we appeal to the proof of Theorem \ref{40}. 
	\item In Section \ref{149} we prove the global Hessian estimate for $\tau< 1$, following arguments of Guan \cite{Guan08}. These estimates apply independently of whether or not $\mu_\Gamma^+>1$.
	\item In Section \ref{150}, we complete the proof of Theorem \ref{55}: we first prove the existence of a unique smooth solution when $\tau<1$ using the continuity method, and we then obtain a Lipschitz viscosity solution in the case $\tau=1$ in the limit as $\tau\rightarrow 1$. 
\end{enumerate}

We point out that, in order to obtain a Lipschitz viscosity solution in the limit $\tau\rightarrow 1$ in Section \ref{150}, it is important that our \textit{a priori} $C^1$ estimates obtained in Sections \ref{146}--\ref{148} are uniform in $\tau\in[0,1]$. On the other hand, the global Hessian estimate in Section \ref{149} deteriorates as $\tau\rightarrow 1$; this is to be expected in view of the work in \cite{LN20b, LNX22}, where the non-existence of $C^2$ solutions is established for all Euclidean domains with disconnected smooth boundary when $\tau=1$.

\subsection{Upper bound}\label{146}

The global upper bound on solutions to \eqref{12} is routine and does not require the assumption $\mu_\Gamma^+>1$: 

\begin{prop}\label{75}
	Suppose $(f,\Gamma)$ satisfies \eqref{21'}--\eqref{24'} and let $\tau\leq 1$. Let $\psi\in C^\infty(M)$ be positive and $\xi\in C^\infty(\partial M)$. Then there exists a constant $C$ which is independent of $\tau$ but dependent on $g_0, f, \Gamma$, a lower bound for $\inf_M \psi $ and an upper bound for $\sup_{\partial M} \xi$ such that any $C^2$ solution to \eqref{12} satisfies $u\leq C$ on $M$. 
\end{prop}
\begin{proof}
	Suppose the maximum of $u$ occurs at $x_0\in M$. If $x_0\in\partial M$, then $u(x_0)\leq \xi(x_0)$. If $x_0\in M\backslash \partial M$, then $\nabla_{g_0}^2 u(x_0) \leq 0$ and $du(x_0)=0$, and hence
	\begin{equation*}
	\psi(x_0)e^{2u(x_0)} \leq f^\tau(-g_0^{-1}A_{g_0})(x_0),
	\end{equation*}
	which yields $u(x_0) \leq \frac{1}{2}\ln\big(\frac{f^\tau(-g_0^{-1}A_{g_0})}{\psi}\big)(x_0)$. 
\end{proof}

\subsection{Lower bound}\label{147}

In this section we obtain the global lower bound on solutions to \eqref{12}:
\begin{prop}\label{151}
	Suppose $(f,\Gamma)$ satisfies \eqref{21'}--\eqref{24'} and \eqref{418}, and let $\tau\leq 1$. Let $\psi\in C^\infty(M)$ be positive and $\xi\in C^\infty(\partial M)$. Then there exists a constant $C$ which is independent of $\tau$ but dependent on $g_0, f, \Gamma$, an upper bound for $\|\psi\|_{C^1(M)}$ and a lower bound for $\inf_{\partial M}\xi$ such that any $C^3$ solution to \eqref{12} satisfies $u\geq C$ on $M$. 
\end{prop}

 There are two main ingredients in our proof of Proposition \ref{151}: our local interior gradient estimate from Theorem \ref{40}, and a lower bound in a uniform neighbourhood of $\partial M$; the assumption $\mu_\Gamma^+>1$ plays a role at both stages. As pointed out before, a delicate point is that we do not assume that the background metric satisfies $\lambda(-g_0^{-1}A_{g_0})\in\Gamma$ on $M$ -- if such an assumption is made, then the proof of the lower bound is as straightforward as the proof of Proposition \ref{75}. In our case, the global lower bound requires more work and is one of the key steps in this paper.

To state our result concerning the lower bound near $\partial M$, for $\delta>0$ we denote 
\begin{align*}
M_\delta = \{x\in M:\operatorname{d}(x,\partial M)<\delta\},
\end{align*}
where $\operatorname{d}(x,\partial M)$ is the distance from $x$ to $\partial M$ with respect to $g_0$. It is well-known that for $\delta>0$ sufficiently small, $M_\delta$ is a tubular neighbourhood of $\partial M$. We show:

\begin{prop}\label{41}
	Under the same hypotheses as Proposition \ref{151}, there exists a constant $\delta>0$ which is independent of $\tau$ but dependent on $g_0, f, \Gamma$, an upper bound for $\sup_{M}\psi$ and a lower bound for $\inf_{\partial M}\xi$ such that any $C^3$ solution $u$ to \eqref{12} satisfies $u\geq \inf_{\partial M}\xi - 1$ in $M_{\delta}$. 
\end{prop}

Assuming the validity of Proposition \ref{41} for now, we give the proof of Proposition \ref{151}:

\begin{proof}[Proof of Proposition \ref{151}]
	Let $\delta>0$ be as in the statement of Proposition \ref{41}, so that $u$ satisfies the lower bound $u\geq \inf_{\partial M}\xi - 1$ in $M_{\delta}$. It follows that
	\begin{align}\label{163}
	u \geq \inf_{\partial M}\xi - 1 - \operatorname{diam}(M,g_0)\sup_{M\backslash M_\delta}|\nabla_{g_0} u|_{g_0} \quad \text{in }M. 
	\end{align}
	On the other hand, by Theorem \ref{40} and the uniform upper bound for $u$ obtained in Proposition \ref{75}, we have
	\begin{align}\label{162}
	|\nabla_{g_0} u|_{g_0} \leq C(\delta^{-1}+1) \quad \text{in }M\backslash M_\delta.
	\end{align}
	Substituting \eqref{162} into \eqref{163}, the proof of Proposition \ref{151} is complete. 
\end{proof}

Roughly speaking, to prove Proposition \ref{41} we cover a neighbourhood of $\partial M$ by small annuli on which we construct suitable comparison functions. The construction of such comparison functions is given in the following proposition (which is a more precise version of Proposition \ref{54'} stated in the introduction). For a Riemannian metric $g_0$ defined on a neighbourhood of the origin in $\mathbb{R}^n$, let $r(x) = \operatorname{d}_{g_0}(0,x)$, let $\mathbb{S}_r = \partial \mathbb{B}_r$ denote the geodesic sphere of radius $r$ centred at the origin, and denote by $A_{r_1, r_2}$ the annulus $\mathbb{B}_{r_2}\backslash\overline{\mathbb{B}}_{r_1}$. We also denote 
\begin{align*}
\beta = \frac{2}{\mu_\Gamma^+ - 1},
\end{align*}
and recall the convention $g_w = e^{2w}g_0$.

\begin{prop}\label{54}
	Suppose $(f,\Gamma)$ satisfies \eqref{21'}--\eqref{24'} and \eqref{418}, let $g_0$ be a Riemannian metric defined on a neighbourhood $\Omega$ of the origin in $\mathbb{R}^n$, and fix a constant $\ep>0$. Then there exists a constant $C>1$ depending only on $g_0, f$ and $\Gamma$, and a constant $0<R<1$ depending additionally on $\ep$, such that for each $m\in\mathbb{R}$, 
	\begin{align}\label{129}
	w(r) \defeq (\beta+\ep)\ln\bigg(\frac{r_+ - r}{r_+ - r_-}\bigg)  +m 
	\end{align}
	satisfies
	\begin{align}\label{76}
	\begin{cases}
	f(\lambda(-g_w^{-1}A_{g_w})) \geq \frac{f(-\mu_\Gamma^+ + C^{-1}\ep,\,1,\, \dots\,,\,1)}{Ce^{2m}(r_+-r_-)^2}>0, \quad \lambda(-g_w^{-1}A_{g_w})\in\Gamma &  \text{on }A_{r_-, r_+} \\
	w(x) = m & \text{for }x\in\mathbb{S}_{r_-} \\
	w(x)\rightarrow -\infty & \text{as }d(x,\mathbb{S}_{r_+})\rightarrow 0
	\end{cases}
	\end{align}
	whenever $1<\frac{r_+}{r_-}<1+\frac{\ep}{2(\beta+2)}$ and $r_+<R$.
\end{prop}

\begin{rmk}
	Our choice of $w$ in \eqref{129} is motivated by the work of Chang, Han \& Yang \cite{CHY05} on radial solutions to the $\sigma_k$-Yamabe equation on annular domains in $\mathbb{R}^n$. Indeed, when $\ep=0$ and $\mu_\Gamma^+ = \frac{n-k}{k}$, \eqref{51} corresponds to the leading order term in the solution to the $\sigma_k$-Yamabe equation in $\Gamma_k^-$ on annular domains in $\mathbb{R}^n$ for $k<\frac{n}{2}$. 
\end{rmk}

\begin{rmk}
	We reiterate that Proposition \ref{54} relies crucially on the assumption $\mu_\Gamma^+>1$, and that a similar construction is not possible when $\mu_\Gamma ^+ \leq 1$ -- see Remark \ref{182} in the introduction. 
\end{rmk}

Assuming the validity of Proposition \ref{54} for now, we first give the proof of Proposition \ref{41} -- the reader may wish to refer to Figure \ref{145} in the following argument:

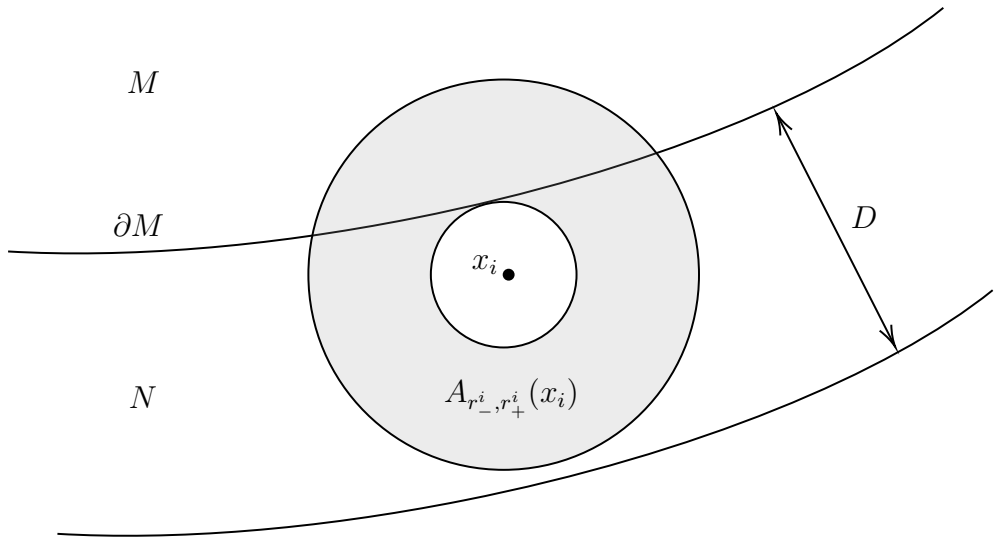
\begin{figure}
	\tikzset{every picture/.style={line width=0.75pt}} 
	
	\begin{tikzpicture}[x=0.75pt,y=0.75pt,yscale=-1,xscale=1]
	\path (0,300); 
	
	\draw  [draw opacity=0] (533.23,29.8) .. controls (480.64,72.17) and (382.66,112.78) .. (266.62,135.88) .. controls (191.19,150.9) and (119.97,156.16) .. (61.55,152.78) -- (244.42,24.36) -- cycle ; \draw   (533.23,29.8) .. controls (480.64,72.17) and (382.66,112.78) .. (266.62,135.88) .. controls (191.19,150.9) and (119.97,156.16) .. (61.55,152.78) ;  
	\draw  [draw opacity=0] (558.11,172.32) .. controls (505.51,214.69) and (407.53,255.3) .. (291.5,278.4) .. controls (216.06,293.42) and (144.84,298.68) .. (86.43,295.3) -- (269.29,166.88) -- cycle ; \draw   (558.11,172.32) .. controls (505.51,214.69) and (407.53,255.3) .. (291.5,278.4) .. controls (216.06,293.42) and (144.84,298.68) .. (86.43,295.3) ;  
	\draw  [fill={rgb, 255:red, 155; green, 155; blue, 155 }  ,fill opacity=0.19 ,even odd rule] (274.76,164.5) .. controls (274.76,144.21) and (291.21,127.76) .. (311.5,127.76) .. controls (331.79,127.76) and (348.24,144.21) .. (348.24,164.5) .. controls (348.24,184.79) and (331.79,201.24) .. (311.5,201.24) .. controls (291.21,201.24) and (274.76,184.79) .. (274.76,164.5)(213,164.5) .. controls (213,110.1) and (257.1,66) .. (311.5,66) .. controls (365.9,66) and (410,110.1) .. (410,164.5) .. controls (410,218.9) and (365.9,263) .. (311.5,263) .. controls (257.1,263) and (213,218.9) .. (213,164.5) ;
	\draw  [fill={rgb, 255:red, 0; green, 0; blue, 0 }  ,fill opacity=1 ] (311.5,164.5) .. controls (311.5,163.12) and (312.62,162) .. (314,162) .. controls (315.38,162) and (316.5,163.12) .. (316.5,164.5) .. controls (316.5,165.88) and (315.38,167) .. (314,167) .. controls (312.62,167) and (311.5,165.88) .. (311.5,164.5) -- cycle ;
	\draw    (449.91,84.78) -- (508.09,199.22) ;
	\draw [shift={(509,201)}, rotate = 243.05] [color={rgb, 255:red, 0; green, 0; blue, 0 }  ][line width=0.75]    (10.93,-3.29) .. controls (6.95,-1.4) and (3.31,-0.3) .. (0,0) .. controls (3.31,0.3) and (6.95,1.4) .. (10.93,3.29)   ;
	\draw [shift={(449,83)}, rotate = 63.05] [color={rgb, 255:red, 0; green, 0; blue, 0 }  ][line width=0.75]    (10.93,-3.29) .. controls (6.95,-1.4) and (3.31,-0.3) .. (0,0) .. controls (3.31,0.3) and (6.95,1.4) .. (10.93,3.29)   ;
	
	\draw (120,60.4) node [anchor=north west][inner sep=0.75pt]    {$M$};
	\draw (113,132.4) node [anchor=north west][inner sep=0.75pt]    {$\partial M$};
	\draw (121,219.4) node [anchor=north west][inner sep=0.75pt]    {$N$};
	\draw (280,215.4) node [anchor=north west][inner sep=0.75pt]    {$A_{r_{-}^i, r_+^i}(x_i)$};
	\draw (294,153.4) node [anchor=north west][inner sep=0.75pt]    {$x_{i}$};
	\draw (485,128.4) node [anchor=north west][inner sep=0.75pt]    {$D $};

	\end{tikzpicture}
	\caption{An annulus in the covering of $\partial M$ in the proof of Proposition \ref{41}.} \label{145}
\end{figure}

\begin{proof}[Proof of Proposition \ref{41}]
	We attach a collar neighbourhood $N$ to $\partial M$ such that $g_0$ extends smoothly to $M\cup N$; we denote this extension also by $g_0$. Let 
	\begin{equation*}
	D = \inf_{x\in \partial M} \operatorname{d}_{g_0}(x,\partial(M\cup N))
	\end{equation*}
	denote the thickness of $N$. Fix $\ep>0$ and let $m = \inf_{\partial M}\xi$, and cover a neighbourhood of $\partial M$ in $M$ by a finite collection of annuli $\{A_{r_-^i, r_+^i}(x_i)\}_{1\leq i \leq K}$ centred at $x_i$ such that the collection $\{A_{r_-^i, \frac{1}{2}(r_-^i + r_+^i)}(x_i)\}$ still covers a neighbourhood of $\partial M$ in $M$, and such that for each $i$:
	\begin{enumerate}
		\item $x_i\in N$,
		\item $r_-^i + r_+^i< D$,
		\item $r_-^i = \operatorname{d}_{g_0}(x_i, \partial M)$,
		\item The closed ball $\overline{B_{r_+^i}(x_i)}$ is contained in a single normal coordinate chart $(U_i, \zeta_i)$ mapping $x_i$ to the origin, 
		\item \begin{align*}\frac{r_+^i}{r_-^i} \leq 1+\frac{\ep}{2(\beta+2)},
		\end{align*} 
		\item $r_i^+ < R$ is sufficiently small so that $\frac{f(-\mu_\Gamma^+ + C^{-1}\ep,\,1,\, \dots\,,\,1)}{Ce^{2m}(r_+^i - r_-^i)^2} \geq \sup_M \psi$ (here $C$ and $R$ are as in the statement of Proposition \ref{54}, where we are implicitly identifying the annulus $A_{r_-^i, r_+^i}^i(x_i)$ with its image under $\zeta_i$, which is possible by Property 4).
	\end{enumerate}
 In what follows, we continue to implicitly make the identification between $A_{r_-^i, r_+^i}^i(x_i)$ and  its image under $\zeta_i$. 

Let $w_i$ denote the solution obtained in Proposition \ref{54} on $A_{r_-^i, r_+^i}(x_i)$ with $\ep>0$ and $m=\inf_{\partial M}\xi$ as fixed above. Since $w_i$ is radially decreasing and $w_i(x) =\inf_{\partial M}\xi$ for $x\in \mathbb{S}_{r_-^i}(x_i)$, we have $w_i \leq \inf_{\partial M}\xi$ on $A_{r_-^i, r_+^i}^i(x_i)\cap\partial M$. On the other hand, $w_i = -\infty < u$ on $\mathbb{S}_{r_+^i}(x_i)$. Therefore, the comparison principle (see Proposition \ref{174} below) yields $u\geq w_i$ on $A_{r_-^i, r_+^i}(x_i)\cap M$ for each $i$. This yields a finite lower bound for $u$ on $A_{r_-^i, \frac{1}{2}(r_-^i + r_+^i)}(x_i)$. Since we assume the collection $\{A_{r_-^i, \frac{1}{2}(r_-^i + r_+^i)}(x_i)\}$ still covers a neighbourhood of $\partial M$ in $M$, we may piece together the estimates for $u$ on each annulus $A_{r_-^i, \frac{1}{2}(r_-^i + r_+^i)}(x_i)$ to obtain the desired estimate for $u$ on a uniform neighbourhood of $\partial M$ in $M$. 
\end{proof}

In the above proof we made use of the following comparison principle:

\begin{prop}[Comparison principle]\label{174}
	Let $\alpha>0$ be a positive constant and $(M,g)$ a compact Riemannian manifold with non-empty boundary $\partial M$. Suppose $u,v \in C^0(M)$ with at least one of $u$ or $v$ belonging to $C^2(M\backslash \partial M)$. If $f(-g_u^{-1}A_{g_u}) \geq f(-g_v^{-1}A_{g_v})\geq \alpha > 0$ in the viscosity sense on $M\backslash \partial M$ and $u\leq v$ on $\partial M$, then $u\leq v$ in $M$. 
\end{prop}

In the proof of Proposition \ref{41}, we only needed Proposition \ref{174} in the case that both $u,v \in C^2(M\backslash \partial M)$. In this case, the proof of Proposition \ref{174} is standard in light of the fact that if $f(-g_v^{-1}A_{g_v})>0$, $c$ is a positive constant and $w = v + c$, then $f(-g_w^{-1}A_{g_w}) < f(-g_v^{-1}A_{g_v})$. The case when $u\in C^0(M)$ in Proposition \ref{174} will be needed later in the paper. When $u\in C^2(M\backslash \partial M)$, Proposition \ref{174} follows from \cite[Theorem 2.1]{CLN13}, since the proof on page 130 therein applies also on Riemannian manifolds with boundary. When $v\in C^2(M\backslash \partial M)$, Proposition \ref{174} again follows from \cite[Theorem 2.1]{CLN13}, therein considering $\widetilde{F}(x,s,p,M) \defeq -F(x,-s,-p,-M)$ in place of $F$.

We now give the proof of Proposition \ref{54}:

\begin{proof}[Proof of Proposition \ref{54}]
It will be more convenient to write our conformal metrics in the form $g^v = v^{-2}g_0$, so that $g_w = g^v$ for $e^{2w}=v^{-2}$. Then the $(0,2)$-Schouten tensor of $g^v$ is given by
\begin{align*}
(A_{g^v})_{ij}= v^{-1}(\nabla_{g_0}^2 v)_{ij} - \frac{1}{2}v^{-2}|\nabla_{g_0} v|_{g_0}^2 (g_0)_{ij} + (A_{g_0})_{ij}. 
\end{align*}
In a fixed normal coordinate system based at the origin, it follows that if $v=v(r)$ then
\begin{align}\label{50}
((g^v)^{-1}A_{g^v})^p_j & =v^2\bigg(\lambda\delta_j^p + \chi\frac{x^px_j}{r^2}\bigg) + O(r^2)v|v_{rr}| + O(r)\big(v^2 + v|v_r|\big) \quad \text{as }r\rightarrow 0,
\end{align}
where 
\begin{equation}\label{16}
\lambda = \frac{v_r}{rv}\bigg(1-\frac{rv_r}{2v}\bigg) \quad\text{and}\quad \chi = \frac{v_{rr}}{v} - \frac{v_r}{vr};
\end{equation}
we refer the reader to Appendix \ref{appb} for the derivation of \eqref{50}. Therefore 
\begin{align}\label{130}
(-(g^v)^{-1}A_{g^v})^p_j \geq -v^2\bigg(\lambda\delta_j^p + \chi\frac{x^px_j}{r^2}\bigg) - |\Psi|\delta_j^p
\end{align}
in the sense of matrices, where $|\Psi| = O(r^2)v|v_{rr}| + O(r)\big(v^2 + v|v_r|\big)$ as $r\rightarrow 0$. \medskip

\noindent \textbf{Step 1:} In this first step we compute and estimate the quantities on the RHS of \eqref{130} for our particular choice of $w$ in \eqref{129}, i.e.~for
\begin{align}\label{51}
v(r) = e^{-\Lambda} (r_+ - r)^{-\beta-\ep},
\end{align}
where we have denoted $\Lambda = m - (\beta+\ep)\ln(r_+ - r_-)$. For shorthand we denote $\phi(r) = r_+ - r$. Then
\begin{align}\label{414}
v_r  =  e^{-\Lambda}(\beta+\ep) \phi^{-\beta - \ep -1} \quad \text{and} \quad 
v_{rr} = e^{-\Lambda}(\beta +\ep)(\beta+\ep+1)\phi^{-\beta - \ep -2},
\end{align}
from which it follows that 
\begin{align*}
\frac{v_r}{rv} = (\beta+\ep) r^{-1}\phi^{-1}, \quad \frac{rv_r}{2v} =  \frac{\beta+\ep}{2}r\phi^{-1} \quad \text{and} \quad \frac{v_{rr}}{v} = (\beta+\ep)(\beta+\ep+1)\phi^{-2}.
\end{align*}
Therefore
\begin{align}\label{133}
\lambda = \frac{v_r}{rv}\bigg(1-\frac{rv_r}{2v}\bigg) =  (\beta+\ep) r^{-1}\phi^{-1}\bigg(1 - \frac{\beta+\ep}{2}r\phi^{-1}\bigg)
\end{align}
and
\begin{align}\label{134}
\chi = \frac{v_{rr}}{v} - \frac{v_r}{vr} = -(\beta+\ep)r^{-1}\phi^{-1} \bigg(1 - (\beta+\ep+1)r\phi^{-1}\bigg).
\end{align}

For $\Psi$ we estimate using \eqref{414} to get
\begin{align*}
|\Psi| & \leq  Cr^2v|v_{rr}| + Cr(v^2 + v|v_r|) \nonumber \\
& \leq Cre^{-2\Lambda}\phi^{-2\beta - 2\ep - 2}\Big(r + \phi + \phi^2\Big) \nonumber \\
& \leq C_1 r e^{-2\Lambda}\phi^{-2\beta - 2\ep - 2} \eqdef \eta. 
\end{align*}

\noindent \textbf{Step 2:} We now use the computations from Step 1 to analyse the eigenvalues of the matrix on the RHS of \eqref{130}, or more precisely the eigenvalues of $-v^2(\lambda\delta_j^p + \chi\frac{x^px_j}{r^2}) - \eta\delta^p_j$, which are given by 
\begin{align*}
-(\chi v^2+\lambda v^2 +\eta, \lambda  v^2 + \eta, \dots, \lambda v^2 +\eta ). 
\end{align*}
We write this vector of eigenvalues more conveniently as
\begin{align*}
(-\lambda v^2 - \eta)\bigg(\frac{\chi v^2}{\lambda v^2 + \eta}+1, 1, \dots, 1\bigg).
\end{align*} 
We make the following two claims: \bigskip

\noindent\textit{\underline{Claim 1:}} There exist constants $c_1>0$ and $0<R_1<1$ depending only on $g_0, f$ and $\Gamma$ such that 
\begin{align}\label{164}
-\lambda v^2 - \eta >c_1e^{-2\Lambda}\phi^{-2\beta - 2\ep - 2} \quad \text{in }\{r_-<r<r_+\}
\end{align}
whenever $1<\frac{r_+}{r_-}< 1+\frac{\ep}{2(\beta+2)}$ and $r_+<R_1$. \bigskip 

\noindent\textit{\underline{Claim 2:}} There exists a constant $c_2>0$ depending only on $g_0, f$ and $\Gamma$, and a constant $0< R_2<1$ depending additionally on $\ep$, such that
\begin{align}\label{202}
\frac{\chi v^2}{\lambda v^2 + \eta}+1 > - \mu_\Gamma^+ + c_2\ep \quad \text{in }\{r_-<r<r_+\}
\end{align}
whenever $1<\frac{r_+}{r_-}< 1+\frac{\ep}{2(\beta+2)}$ and $r_+<R_2$. \bigskip 

Once the claims are proved, Proposition \ref{54} is obtained as follows. First fix $r_+$ and $r_-$ such that $1<\frac{r_+}{r_-}< 1+\frac{\ep}{2(\beta+2)}$ and $r_+<\min\{R_1,R_2\}$. By Claim 2 and the definition of $\mu_{\Gamma}^+$, 
\begin{align*}
f\bigg(\frac{\chi v^2}{\lambda v^2 + \eta}+1, 1, \dots, 1\bigg)> f(-\mu_\Gamma^+ + c_2\ep, 1,\dots,1)>0 \quad \text{in }\{r_-<r<r_+\}.
\end{align*}
Then by Claim 1, it follows that 
\begin{align*}
f\bigg((-\lambda v^2 - \eta)\bigg(&\frac{\chi v^2}{\lambda v^2 + \eta}+1, 1, \dots, 1\bigg)\bigg)\nonumber \\
& >c_1e^{-2\Lambda}\phi^{-2\beta-2\ep - 2}f(-\mu_\Gamma^+ + c_2\ep, 1,\dots,1) \quad \text{in }\{r_-<r<r_+\},
\end{align*} 
from which \eqref{76} follows. To complete the proof of Proposition \ref{54}, it therefore remains to prove Claims 1 and 2. \medskip 

\noindent\textbf{Note:} We will use at various stages the fact that 
\begin{align}\label{156}
1<\frac{r_+}{r_-}< 1+\frac{\ep}{2(\beta+2)} \quad \iff \quad 0<\phi r^{-1} < \frac{\ep}{2(\beta+2)} \quad \text{in }\{r_-<r<r_+\}.
\end{align}

\begin{proof}[Proof of Claim 1]
	Suppose $1<\frac{r_+}{r_-}< 1+\frac{\ep}{2(\beta+2)}$ and $r_+<1$. We start by computing
	\begin{align}\label{200}
	-\lambda v^2 & =  e^{-2\Lambda}\phi^{-2\beta - 2\ep -2}(\beta+\ep)\bigg( \frac{\beta+\ep }{2} - \phi r^{-1}\bigg).
	\end{align} 
	By \eqref{156} and \eqref{200}, it follows that
	\begin{align}\label{157}
	-\lambda v^2 \geq \frac{1}{C}e^{-2\Lambda}\phi^{-2\beta - 2\ep - 2} \quad \text{in }\{r_-<r<r_+\}.
	\end{align}
Recalling also that
	\begin{align}\label{201}
	\eta = C_1re^{-2\Lambda}\phi^{-2\beta - 2\ep -2},
	\end{align}
	we see that \eqref{157} and \eqref{201} imply 
\begin{align}\label{77}
-\lambda v^2 - \eta \geq (C^{-1}-C_1r)e^{-2\Lambda}\phi^{-2\beta- 2\ep -2} \quad \text{in }\{r_- < r < r_+\}. 
\end{align}
The inequality \eqref{164} then follows from \eqref{77} after taking $r_+$ sufficiently small. This completes the proof of Claim 1.
\end{proof}

\begin{proof}[Proof of Claim 2]
	Suppose $1<\frac{r_+}{r_-}< 1+\frac{\ep}{2(\beta+2)}$ and $r_+<1$. By \eqref{200} and the fact that $\mu_\Gamma^+ = \frac{2+\beta}{\beta}$ we have 
	\begin{align}\label{423}
	-\lambda v^2 - \mu_\Gamma^+\lambda v^2 = -\frac{2+2\beta}{\beta}\lambda v^2 = \frac{2+2\beta}{\beta}e^{-2\Lambda}\phi^{-2\beta - 2\ep -2}(\beta+\ep)\bigg( \frac{\beta+\ep }{2} - \phi r^{-1}\bigg),
	\end{align}
	and by the formula for $\chi$ in \eqref{134} we have
	\begin{align}\label{424}
	-\chi v^2 = e^{-2\Lambda}(\beta+\ep)\phi^{-2\beta-2\ep-2}\big(\phi r^{-1} - (\beta+\ep+1)\big).
	\end{align}
	It follows from \eqref{423} and \eqref{424} that 
	\begin{align}\label{159}
	-\chi v^2 - \lambda v^2 - \mu_\Gamma^+\lambda v^2 = e^{-2\Lambda}\phi^{-2\beta - 2\ep - 2}\frac{\beta+\ep}{\beta}\big(\ep-(\beta+2)r^{-1}\phi\big).
	\end{align}
	On the other hand, by \eqref{156} we have
	\begin{align*}
	\frac{\beta+\ep}{\beta}\Big(\ep-(\beta+2)r^{-1}\phi\Big) > \frac{\ep}{2} \quad \text{in }\{r_- < r< r_+\},
	\end{align*}
	which when substituted into \eqref{159} yields
	\begin{align}\label{160}
	-\chi v^2 - \lambda v^2 - \mu_\Gamma^+ \lambda v^2 > \frac{\ep}{2}e^{-2\Lambda}\phi^{-2\beta - 2\ep - 2} \quad \text{in }\{r_- < r< r_+\}.
	\end{align}
	Recalling \eqref{201}, the estimate \eqref{160} therefore implies
	\begin{align}\label{44}
	-\chi v^2 -\lambda v^2 -\mu_\Gamma^+&\lambda v^2 - \eta - \mu_\Gamma^+ \eta \nonumber \\
	& \geq \bigg(\frac{\ep}{2}- Cr\bigg)e^{-2\Lambda}\phi^{-2\beta-2\ep - 2} \quad \text{in }\{r_- < r< r_+\}.
	\end{align}
	After taking $r_+$ smaller if necessary (but in a way that only depends on $\ep$ and the constant $C$ in \eqref{44}), we therefore have
	\begin{align*}
	-\chi v^2 -\lambda v^2 -\mu_\Gamma^+\lambda v^2 - \eta - \mu_\Gamma^+ \eta \geq \frac{\ep}{4}e^{-2\Lambda}\phi^{-2\beta - 2\ep - 2}\quad \text{in }\{r_- < r< r_+\},
	\end{align*}
	or equivalently 
\begin{align}\label{203}
\frac{\chi v^2}{\lambda v^2 + \eta}+1 \geq -\mu_\Gamma^++ \frac{\frac{\ep}{4}e^{-2\Lambda}\phi^{-2\beta-2\ep-2}}{-\lambda v^2 - \eta}. 
\end{align}
On the other hand, by \eqref{200} we have
\begin{align*}
0<-\lambda v^2 - \eta  \leq - \lambda v^2 \leq Ce^{-2\Lambda}\phi^{-2\beta-2\ep-2}\quad\text{in }\{r_- < r< r_+\}.
\end{align*}
Thus, if $r_+$ is chosen sufficiently small (but depending only on $g_0, f, \Gamma$ and $\ep$), we see
\begin{align*}
\frac{\chi v^2}{\lambda v^2 + \eta}+1 \geq -\mu_\Gamma^+ + c\ep \quad \text{in }\{r_- < r< r_+\},
\end{align*}
as required. This completes the proof of Claim 2.
\end{proof} 
\noindent As explained above, with Claims 1 and 2 established, the proof of Proposition \ref{54} is complete. 
\end{proof}

\subsection{Gradient estimate}\label{148}

In this section we prove the global gradient estimate:

\begin{prop}\label{154}
	Suppose $(f,\Gamma)$ satisfies \eqref{21'}--\eqref{24'} and \eqref{418}, and let $\tau\leq 1$. Let $\psi\in C^\infty(M)$ be positive and $\xi\in C^\infty(\partial M)$. Then there exists a constant $C$ which is independent of $\tau$ but dependent on $g_0, f, \Gamma$ and upper bounds for $\|\psi\|_{C^1(M)}, \|\xi\|_{C^2(\partial M)}$ and $\|u\|_{C^0(M)}$ such that any $C^3$ solution to \eqref{12} satisfies $|\nabla_{g_0} u|_{g_0}\leq C$ on $M$. 
\end{prop}
\begin{proof}
	By our interior local gradient estimate in Theorem \ref{40}, we only need to prove the gradient estimate near the boundary, say in $B_{1/2}(y_0)\cap M$ where $y_0\in \partial M$ is arbitrary. Consider $H=\rho|\nabla_{g_0} u|_{g_0}^2$, where $\rho$ is a smooth cutoff function satisfying $\rho=1$ on $B_{1/2}(y_0)$, $\rho=0$ outside $B_1(y_0)$, $|\nabla_{g_0} \rho|_{g_0}\leq C\rho^{1/2}$ and $|\nabla_{g_0}^2\rho|_{g_0}\leq C$. Suppose that $H$ attains its maximum at $x_0\in M$. If $x_0\not\in B_1(y_0)\cap M$, then $\nabla_{g_0} u=0$ in $B_{1/2}(y_0)\cap M$ and we are done. If $x_0\in B_1(y_0)\cap (M\backslash \partial M)$, then our proof of Theorem \ref{40} applies and we again obtain the desired estimate. It remains to consider the case that $x_0\in B_1(y_0)\cap \partial M$.

	We first observe that, since the tangential derivatives of $u$ on $\partial M$ are bounded by $\|\xi\|_{C^1(\partial M)}$, we only need to bound the normal derivative $\nabla_{\nu}u(x_0)$, where $\nu$ denotes the inward pointing unit normal to $\partial M$ at $x_0$. We first consider the lower bound for $\nabla_{\nu}u(x_0)$. With the same setup and notation as in the proof of Proposition \ref{41}, except now with $m=u(x_0)$ and $\ep>0$ to be fixed later, consider the set $\mathcal{S}$ of annuli $A_{r_-, r_+}(y)$ satisfying $\mathbb{S}_{r_-}(y)\cap\partial M = \{x_0\}$ and conditions 1)--6) in the proof of Proposition \ref{41}. By taking $\ep$ sufficiently large and $r_+$ sufficiently small, we may choose an annulus $A_{r_-, r_+}(y)$ in $\mathcal{S}$ such that the function $w$ on $A_{r_-, r_+}(y)$, as defined in \eqref{129}, satisfies $w \leq u$ on $A_{r_-,r_+}(y)\cap \partial M$. By the comparison principle stated in Proposition \ref{174}, it follows that $w\leq u$ on $A_{r_-,r_+}(y) \cap M$. Thus for $x\in A_{r_-,r_+}(y) \cap M$ we have
	\begin{align*}
	\frac{u(x) - u(x_0)}{\operatorname{d}(x,x_0)} = \frac{u(x) - w(x_0)}{\operatorname{d}(x,x_0)} \geq \frac{w(x) - w(x_0)}{\operatorname{d}(x,x_0)},
	\end{align*}
	which implies $\nabla_\nu u(x_0) \geq \nabla_\nu w(x_0)$. 
	
	For the upper bound for $\nabla_\nu u(x_0)$, we construct a barrier function similar to that of Guan \cite{Guan08}. First observe that since $\Gamma\subset\Gamma_1^+$, we have
	\begin{align*}
	0 < \sigma_1(-g_0^{-1}A_{g_u}) & = \Delta_{g_0} u + \frac{n-2}{2}|\nabla_{g_0} u|_{g_0}^2 -\sigma_1(g_0^{-1}A_{g_0}).  
	\end{align*}
	Now let $d(x) = \operatorname{d}(x,\partial M)$ and recall $M_{\delta} = \{x\in M: d(x)<\delta\}$. It is well-known that for sufficiently small $\delta>0$, $d$ is smooth in $M_\delta$ with $|\nabla_{g_0} d|_{g_0} = 1$. To obtain an upper bound for $\nabla_\nu u(x_0)$, it suffices to find a function $\bar{u}\in C^3(M_\delta)$ satisfying 
	\begin{equation}\label{10'}
	\begin{cases}
	\sigma_1(-g_0^{-1}A_{g_{\bar{u}}}) \leq 0 & \text{in }M_\delta\\
	\bar{u} = u  & \text{on }\partial M\\
	\bar{u}\geq u & \text{on }\partial M_\delta \backslash \partial M. 
	\end{cases}
	\end{equation}
	Indeed, once such a function $\bar{u}$ is obtained, the maximum principle implies $\bar{u}\geq u$ on $M_\delta$, and it follows that for any $x\in M_\delta$, we have
	\begin{equation*}
	\frac{u(x) - u(x_0)}{\operatorname{d}(x,x_0)} = \frac{u(x) - \bar{u}(x_0)}{\operatorname{d}(x,x_0)} \leq \frac{\bar{u}(x) - \bar{u}(x_0)}{\operatorname{d}(x,x_0)},
	\end{equation*}
	which implies $\nabla_\nu u(x_0) \leq \nabla_\nu \bar{u}(x_0)$. 
	
	To construct $\bar{u}\in C^3(M_\delta)$ satisfying \eqref{10'}, we first extend $\xi$ to a smooth function $\overline{\xi}$ on $M_\delta$ by defining $\overline{\xi}$ to be constant along geodesics normal to $\partial M$; such a construction is always possible for sufficiently small $\delta>0$. We then define
	\begin{equation*}
	\bar{u}(x) = \overline{\xi} + \frac{1}{n-2} \ln \frac{d(x)+\delta^2}{\delta^2}.
	\end{equation*}
	We first observe that $\bar{u}|_{\partial M} = \overline{\xi}|_{\partial M} = \xi = u|_{\partial M}$. Next we calculate $\sigma_1(-g_0^{-1}A_{g_{\bar{u}}})$. In what follows, we denote by $\nabla d$ the differential of $d$ (whereas $\nabla_{g_0}d$ will continue to denote the gradient of $d$ with respect to $g_0$). Routine computations yield
	\begin{align*}
	\nabla_{g_0}\bar{u}(x) = \nabla_{g_0}\overline{\xi}(x) + \frac{1}{n-2}\frac{\nabla_{g_0}d(x)}{d(x) +\delta^2} 
	\end{align*} 
	and
	\begin{align*} 
	 \nabla_{g_0}^2 \bar{u}(x) = \nabla_{g_0}^2 \overline{\xi}(x) + \frac{1}{n-2}\bigg(\frac{ \nabla_{g_0}^2 d(x) }{d(x) + \delta^2} - \frac{\nabla d(x)\otimes \nabla d(x)}{(d(x)+\delta^2)^2}\bigg), 
	\end{align*} 
	from which it follows that 
	\begin{align}\label{74}
	\sigma_1(-g_0^{-1}A_{g_{\bar{u}}}) & =  \Delta_{g_0} \bar{u} + \frac{n-2}{2}|\nabla_{g_0}\bar{u}|_{g_0}^2 -\sigma_1(g_0^{-1}A_{g_0}) \nonumber \\
	& \leq C - \frac{1}{2(n-2)}\frac{1}{(d(x) + \delta^2)^2} + \frac{C}{d(x) + \delta^2}, 
	\end{align}
	where we have used the fact that $|\nabla_{g_0}d|_{g_0} = 1$ and $|\Delta_{g_0}d|\leq C$ in $M_{\delta}$ for $\delta$ sufficiently small, and the fact that $\|\overline{\xi}\|_{C^2(M_\delta)}$ is bounded by a constant depending only on $g_0$ and $\|\xi\|_{C^2(\partial M)}$. We then see that the negative term on the last line of \eqref{74} dominates the remaining terms for $\delta>0$ sufficiently small. Therefore, for $\delta>0$ sufficiently small, we have $\sigma_1(-g_0^{-1}A_{g_{\bar{u}}})\leq 0$ in $M_\delta$.

	Finally, we observe that on $\partial M_\delta\backslash \partial M$ we have
	\begin{align*}
	\bar{u} = \overline{\xi} +  \frac{1}{n-2}\ln \bigg(\frac{\delta+\delta^2}{\delta^2}\bigg) \geq \overline{\xi} + \frac{1}{n-2}\ln (1/\delta).
	\end{align*}
	Choosing $\delta$ smaller if necessary so that $\overline{\xi} + \frac{1}{n-2}\ln (1/\delta) \geq \max_{M} u$ on $\partial M_\delta\backslash\partial M$, the construction of $\bar{u}$ is complete. This completes the proof of Proposition \ref{154}.
\end{proof}

\subsection{Hessian estimate}\label{149}

In this section we give the global Hessian estimate assuming $\tau<1$:

\begin{prop}\label{83}
		Suppose $(f,\Gamma)$ satisfies \eqref{21'}--\eqref{24'} and let $\tau<1$. Let $\psi\in C^\infty(M)$ be positive and $\xi\in C^\infty(\partial M)$. Then there exists a constant $C$ depending on $g_0, f, \Gamma,(1-\tau)^{-1}$ and upper bounds for $\|\psi\|_{C^2(M)}, \|\xi\|_{C^2(M)}$ and $\|u\|_{C^1(M)}$ such that any solution to \eqref{12} satisfies $|\nabla_{g_0}^2 u|_{g_0} \leq C$ on $M$. 
\end{prop}

We point out that we do not require $\mu_\Gamma^+>1$ in Proposition \ref{83}. 

\begin{proof}
	If the maximum of $|\nabla^2_{g_0} u|_{g_0}$ occurs in $M\backslash \partial M$, then one can appeal to the proof of the global estimate of Gursky \& Viaclovsky \cite{GV03b} if $f=\sigma_k^{1/k}$, or the proof of the global estimate of Guan \cite{Guan08} for general $(f,\Gamma)$ satisfying \eqref{21'}--\eqref{24'}. So we suppose that the maximum occurs at a point $x_0\in\partial M$. Let $e_n$ denote the interior unit normal vector field on $\partial M$, and fix an orthonormal frame $\{e_1,\dots,e_{n-1}\}$ for the tangent bundle of $\partial M$ near $x_0$. By parallel transporting along geodesics normal to $\partial M$, we may extend this to an orthonormal frame $\{e_1,\dots,e_n\}$ for the tangent bundle of $M$ near $x_0$. Since $(\nabla_{g_0}^2 u)_{ij}(x_0) = (\nabla_{g_0}^2\xi)_{ij}(x_0)$ for $i,j\not=n$, we only need to estimate $(\nabla_{g_0}^2 u)_{ij}(x_0)$ when at least one of $i$ or $j$ are equal to $n$. The proof is almost identical to that of Guan in \cite{Guan08}, but for the convenience of the reader we summarise the argument here. In what follows, all computations are carried out in a neighbourhood of $x_0$ on which the frame $\{e_1,\dots,e_n\}$ is defined.

	Still with the convention $g_u = e^{2u}g_0$, it will be convenient to write the equation in \eqref{12} in the equivalent form
	\begin{align}\label{417}
	f(\lambda(-g_0^{-1}A_{g_u}^\tau)) = \psi e^{2u}, \quad \lambda(-g_u^{-1}A_{g_u}^\tau)\in\Gamma  \quad \text{on }M\backslash \partial M,
	\end{align}
	where 
	\begin{align*}
	A_{g_u}^\tau & = \tau A_{g_u} + (1-\tau)\sigma_1(-g_u^{-1}A_{g_u})g_u \nonumber \\
	& = -\tau\nabla_{g_0}^2 u - (1-\tau)\Delta_{g_0} u\,g_0 - b_{n,\tau}|\nabla_{g_0} u|_{g_0}^2 g_0 + \tau du\otimes du + A_{g_0}^\tau
	\end{align*}
	and $b_{n,\tau} = \frac{1}{2}\big(n-2-(n-3)\tau\big)$. Denoting $F[u] = f(\lambda(-g_0^{-1}A_{g_u}^\tau))$ and 
	\begin{align*}
	F^{ij} = \frac{\partial f}{\partial A_{ij}}\bigg|_{A=-g_0^{-1}A_{g_u}^\tau},
	\end{align*}
	the linearisation of $F$ at $u$ in the direction $\eta$ (excluding zeroth order terms) is given by 
	\begin{align}\label{165}
	\mathcal{L}\eta & = F^{ij}\bigg(\tau(\nabla_{g_0}^2 \eta)_{ij} +(1-\tau)\Delta_{g_0} \eta \,(g_0)_{ij} + 2b_{n,\tau}\langle \nabla_{g_0} u,\nabla_{g_0} \eta\rangle_{g_0}(g_0)_{ij} - 2\tau \partial_i u\,\partial_j\eta\bigg) \nonumber \\
	& = F^{ij}\bigg(\tau(\nabla_{g_0}^2 \eta)_{ij} - 2\tau\partial_i u\,\partial_j\eta \bigg) + \bigg((1-\tau)\Delta_{g_0} \eta  + 2b_{n,\tau}\langle \nabla_{g_0} u,\nabla_{g_0} \eta\rangle_{g_0}\bigg)\sum_iF^{ii}.
	\end{align}

	 Now suppose $\delta>0$ is sufficiently small so that $d(x) = \operatorname{d}(x,\partial M)$ is smooth in $M_\delta= \{x\in M: d(x)<\delta\}$. For a positive constant $N$ to be determined later, define
	\begin{align}\label{82}
	v = \frac{N}{2}d^2 - d. 
	\end{align}
	A routine compute shows that for $\delta>0$ sufficiently small, 
	\begin{align}\label{80}
	|\mathcal{L}d| \leq C_0\sum_i F^{ii} \quad\text{in }M_\delta,
	\end{align}
	where $C_0$ is a constant independent of $\tau$ but depending on $g_0$ and an upper bound for $\|u\|_{C^1(M)}$. It follows that 
	\begin{align}\label{81}
		\mathcal{L}d^2 & = 2d \mathcal{L}d + 2(1-\tau)|\nabla_{g_0} d|_{g_0}^2\sum_i F^{ii} + 2F^{ij}\partial_i d\,\partial_j d \nonumber \\
	& \geq  2d \mathcal{L}d + 2(1-\tau)\sum_i F^{ii} \nonumber \\
	& \geq 2\big((1-\tau)-C_0d\big)\sum_i F^{ii} \quad \text{in }M_\delta.
	\end{align}
Choosing $N\geq \frac{4(1+C_0)}{1-\tau}$ and subsequently $\delta \leq \min\{N^{-1}, C_0^{-1}\}$, one sees from \eqref{80} and \eqref{81} that the function $v$ defined in \eqref{82} satisfies
	\begin{align}\label{152}
	\mathcal{L}v \geq \sum_i F^{ii} \quad \text{and} \quad 
	v \leq -\frac{d}{2}\quad\text{in }M_\delta. 
	\end{align}

	With \eqref{152} in hand, one can then show: 
	\begin{lem}\label{84}
		Fix $\delta>0$ sufficiently small as in the foregoing argument. If $h\in C^2(\overline{M_\delta})$ satisfies $h\leq 0$ on $\partial M$, $h(z_0) = 0$ for some $z_0\in \partial M$ and 
		\begin{align}\label{85}
		-\mathcal{L}h \leq C_1\sum_i F^{ii} \quad \text{in }M_\delta
		\end{align}
		for some constant $C_1$, then 
		\begin{align}\label{86}
		(\nabla_{g_0} h)_n(z_0) \leq C,
		\end{align}
		where $C$ is a constant depending on $g_0$, $C_1$, $(1-\tau)^{-1}$ and upper bounds for $\|h\|_{C^0(\overline{M_\delta})}$ and $\|u\|_{C^1(M)}$. 
	\end{lem}
\begin{proof}
	It is clear from the definition of $v$ that we can choose $A>0$ large (depending on $\|h\|_{C^0(\overline{M_\delta})})$ such that $-Av - h \geq 0$ on $\partial M_\delta$. On the other hand, using \eqref{152} and \eqref{85},  we have
	\begin{align*}
	\mathcal{L}(-Av - h)\leq  (-A+C_1)\sum_i F^{ii} \quad \text{in }M_\delta,
	\end{align*}
	and hence $\mathcal{L}(-Av - h) \leq 0$ in $M_\delta$ for $A$ sufficiently large. Thus, for $A$ sufficiently large the maximum principle yields $-Av - h \geq 0$ in $M_\delta$, and since $(-Av - h)(z_0) = 0$, it follows that $(\nabla_{g_0}(-Av - h))_n(z_0) \geq 0$, i.e.~$(\nabla_{g_0} h)_n(z_0) \leq -A(\nabla_{g_0} v)_n(z_0)$. The estimate \eqref{86} then follows. 
\end{proof}

We now continue the proof of Proposition \ref{83}. Suppose $i\in \{1,\dots,n-1\}$ and define $h=\pm(\nabla_{g_0}(u-\overline{\xi}))_i$, where (as in the proof of Proposition \ref{154}) $\overline{\xi}$ denotes the extension of $\xi$ to $M_\delta$ such that $\overline{\xi}$ is constant along geodesics normal to $\partial M$. By differentiating the equation \eqref{417}, one can show directly that $|\mathcal{L}(\nabla_{g_0} u)_i| \leq C\sum_i F^{ii}$, and by \eqref{183} we also have $|\mathcal{L}\overline{\xi}| \leq C \leq C\sum_i F^{ii}$. Therefore $h$ satisfies the assumptions of Lemma \ref{84}, and it follows from Lemma \ref{84} that
\begin{align*}
|(\nabla_{g_0}^2 u)_{in}(x_0)| \leq C.
\end{align*}

It remains to estimate the double normal derivative $(\nabla_{g_0}^2 u)_{nn}(x_0)$. Note that since $\{e_1,\dots,e_{n}\}$ is an orthonormal frame and $(\nabla_{g_0}^2 u)_{ii}(x_0) = (\nabla_{g_0}^2\xi)_{ii}(x_0)$ for $i\in\{1,\dots,n-1\}$, to obtain an upper (resp.~lower) bound for $(\nabla_{g_0}^2 u)_{nn}(x_0)$, it is equivalent to obtain an upper (resp.~lower) bound for $\Delta_{g_0} u(x_0)$. Now, since $\Gamma\subseteq \Gamma_1^+$, the lower bound $\Delta_{g_0} u \geq -C$ in $M$ is immediate. To obtain the upper bound for $(\nabla_{g_0}^2 u)_{nn}(x_0)$, we may assume $(\nabla_{g_0}^2 u)_{nn}(x_0) \geq 1$, otherwise we are done. We may also assume that with respect to the frame $\{e_1,\dots,e_n\}$, the Hessian of $u$ at $x_0$ is given by $\nabla_{g_0}^2 u(x_0) = \operatorname{diag}((\nabla_{g_0}^2 u)_{11}(x_0),\dots, (\nabla_{g_0}^2 u)_{nn}(x_0))$. Then by the equation \eqref{417}, monotonicity of $f$ and our estimates for $(\nabla_{g_0}^2 u)_{ij}(x_0)$ when $i$ and $j$ are not both equal to $n$, we have
\begin{align}\label{425}
\psi(x_0)e^{2u(x_0)} = f(-g_0^{-1}A_{g_u}^\tau(x_0)) \geq f\big((1-\tau)(\nabla_{g_0}^2 u)_{nn}(x_0)g_0 + B\big),
\end{align}
where $B$ is a symmetric matrix bounded in terms of $\|u\|_{C^1(M)}$. The upper bound for $(\nabla_{g_0}^2 u)_{nn}(x_0)$ then follows from \eqref{425} and homogeneity of $f$. 
\end{proof}

\subsection{Proof of Theorem \ref{55}}\label{150}

We now complete the proof of Theorem \ref{55}:

\begin{proof}[Proof of Theorem \ref{55}]
We first prove the existence of a smooth solution to \eqref{12} when $\tau<1$. Fix $\ep>0$ and let $S_\ep= \{\tau\in[0,1-\ep]\,:\, \eqref{12} \text{ admits a solution in }C^{2,\alpha}(M)\}$. Since \eqref{12} admits a unique smooth solution when $\tau=0$, $S_\ep$ is non-empty. A computation as in \eqref{165} (but now including zeroth order terms) shows that the linearised operator is invertible as a mapping from $C^{2,\alpha}(M)$ to $C^\alpha(M)$, from which openness of $S_\ep$ follows. By Propositions \ref{75} and \ref{151}, solutions to \eqref{12} admit a global $C^0$ estimate. By Proposition \ref{154}, solutions to \eqref{12} therefore admit a global $C^1$ estimate. Note that, at this point, the estimates are independent of $\ep$. By Proposition \ref{83}, one then obtains the global $C^2$ estimate on solutions to \eqref{12}, which do now depend on $\ep$. With the $C^2$ estimate established, \eqref{12} becomes uniformly elliptic, and the regularity theory of Evans-Kyrlov \cite{Ev82, Kry82, Kry83} then implies a $C^{2,\alpha}$ estimate. Thus $S_\ep$ is also closed, and so $S_\ep = [0,1-\ep]$. Since $\ep>0$ was arbitrary, existence of a $C^{2,\alpha}$ solution to \eqref{12} for any $\tau<1$ then follows. Higher regularity then follows from classical Schauder theory, and uniqueness is a consequence of the comparison principle in Proposition \ref{174}. 

Now, since the solutions obtained to \eqref{12} are uniformly bounded in $C^1(M)$ as $\tau\rightarrow 1$, along a sequence $\tau_i\rightarrow 1$ these solutions converge uniformly to some $u\in C^{0,1}(M)$. The proof that $u$ is a viscosity solution to \eqref{12} when $\tau=1$ is exactly the same as in the proof of Theorem 1.3 in \cite{LN20b}, and is omitted here.
\end{proof}

\section{Proof of Theorem \ref{A'}: the fully nonlinear Loewner-Nirenberg problem}\label{135}

In this section we prove Theorem \ref{A'}. Our proof proceeds according to the following steps: 

\begin{enumerate}
	\item In Section \ref{137} we construct a smooth solution to \eqref{-113} when $\tau<1$. The solution is obtained as the limit of solutions with constant finite boundary data $m\in\mathbb{R}$ (which we know to exist by Theorem \ref{55}) as $m\rightarrow\infty$. 
	\item In Section \ref{138} we prove that there exists a smooth solution $u$ to \eqref{-113} when $\tau<1$ satisfying the asymptotics stated in \eqref{114}. 
	\item In Section \ref{139} we prove that \textit{any} smooth solution to \eqref{-113} must satisfy \eqref{114} when $\tau<1$. When combined with the maximum principle, this will imply that the solution $u$ obtained to \eqref{-113} is unique when $\tau<1$. 
	\item In Section \ref{166} we complete the proof of Theorem \ref{A'}.
\end{enumerate}

\subsection{Existence of a smooth solution to \eqref{-113} when $\tau<1$}\label{137}

Fix $\tau<1$ and suppose that $(f,\Gamma)$ satisfies \eqref{21'}--\eqref{24'}, \eqref{25'} and \eqref{418}. By Theorem \ref{55}, we know that for each $m\in\mathbb{R}$, there exists a unique smooth solution $u_m$ to
\begin{equation}\label{13''}
\begin{cases}
f^\tau(\lambda(-g_{u_m}^{-1}A_{g_{u_m}})) = 1, \quad \lambda(-g_{u_m}^{-1}A_{g_{u_m}})\in\Gamma^\tau & \text{on }M\backslash\partial M \\
u_m = m & \text{on }\partial M.
\end{cases}
\end{equation}
In this section we show that in the limit $m\rightarrow\infty$, one obtains a smooth solution $u$ to \eqref{-113}.

\begin{prop}\label{90}
	Fix $\tau<1$ and suppose that $(f,\Gamma)$ satisfies \eqref{21'}--\eqref{24'}, \eqref{25'} and \eqref{418}. Let $u_m$ denote the unique smooth solution to \eqref{13''}. Then a subsequence of $\{u_m\}_m$ converges uniformly as $m\rightarrow \infty$ to a solution $u\in C^\infty(M\backslash\partial M)$ of \eqref{-113}. Moreover, given any constant $\alpha>0$, there exists a constant $\delta>0$ independent of $\tau$ but dependent on $g_0, \alpha, f$ and $\Gamma$ such that $u \geq \alpha$ in $M_\delta\backslash\partial M$. 
\end{prop}

\begin{proof} 
	Since the comparison principle in Proposition \ref{174} implies $u_{m+1} \geq u_m$, to prove the existence of a limit $u\in C^\infty(M\backslash \partial M)$ solving \eqref{-113}, it suffices to show that for each compact set $K\subset M\backslash \partial M$ there exists a constant $C$ independent of $m$ such that $\|u_m\|_{C^2(K)} \leq C$; higher order estimates then follow from the work of Evans-Krylov \cite{Ev82, Kry82} and classical Schauder theory. 
	
	The lower bound is trivial (and in fact global), since $u_m \geq u_1$ for all $m$. Next we address the local upper bound -- note that whilst we obtained a global upper bound in Proposition \ref{75}, the bound therein depends on $m$, which is insufficient for our current purposes. Recalling the normalisation $f(\frac{1}{2}, \dots, \frac{1}{2}) =1$, we have by concavity and homogeneity of $f$ 
	\begin{align}\label{98}
	f(\lambda) \leq f\bigg(\frac{\sigma_1(\lambda)}{n}e\bigg) + \nabla f\bigg(\frac{\sigma_1(\lambda)}{n}e\bigg)\cdot\bigg(\lambda - \frac{\sigma_1(\lambda)}{n}e\bigg) = \frac{f(e)}{n}\sigma_1(\lambda) = \frac{2}{n}\sigma_1(\lambda) \quad \text{for }\lambda\in\Gamma,
	\end{align} 
	and thus any solution to the equation in \eqref{13''} satisfies $R_{g_{u_m}} \leq -n(n-1)$. On the other hand, by the work of Aviles \& McOwen \cite{AM88}, there exists a smooth metric $g_w = e^{2w}g_0$ satisfying
	\begin{align}\label{185}
	\begin{cases}
	R_{g_{w}} = -n(n-1) & \text{on }M\backslash\partial M \\
	w(y)\rightarrow+\infty & \text{as } \operatorname{d}(y,\partial M)\rightarrow 0. 
	\end{cases}
	\end{align}
	By the comparison principle for the semilinear equation \eqref{185}, $u_m \leq w$ in $M\backslash\partial M$ for each $m$, which yields a finite upper bound for $u_m$ on any compact subset of $M\backslash\partial M$ which is independent of $m$. The local gradient estimate then follows from Theorem \ref{40}, or alternatively one can appeal to \cite[Theorem 2.1]{Guan08} since we have the two-sided $C^0$ bound at this point. For the local Hessian estimate, we appeal to \cite[Theorem 3.1]{Guan08}. We therefore obtain the full $C^2$ estimate $\|u_m\|_{C^2(K)} \leq C(K)$ on any compact set $K\subset M\backslash \partial M$, as required. 
	
	It remains to prove the second assertion in the statement of Proposition \ref{90}. Fix $\alpha>0$ and consider the solution $u_{\alpha+1}$ to \eqref{13''} with $m=\alpha+1$. Since $u_{\alpha+1}$ admits a global $C^0$ estimate depending only $g_0, \alpha, f$ and $\Gamma$, there exists a constant $\delta>0$ depending only on $g_0, \alpha, f$ and $\Gamma$ such that $u_{\alpha+1} \geq \alpha$ in $M_\delta$. By the comparison principle in Proposition \ref{174}, $u\geq u_{\alpha+1}$ in $M\backslash\partial M$, and in particular $u \geq \alpha$ in $M_\delta\backslash\partial M$, as required. 
\end{proof}

\subsection{Asymptotics}\label{138}

Fix $\tau<1$ and suppose that $(f,\Gamma)$ satisfies \eqref{21'}--\eqref{24'}, \eqref{25'} and \eqref{418}. In this section we show that there exists a smooth solution $u$ to \eqref{-113} satisfying \eqref{114}, that is
\begin{align}\label{114'}
\lim_{\operatorname{d}(x,\partial M)\rightarrow 0}\big(u(x) + \ln \operatorname{d}(x,\partial M)\big) = 0. 
\end{align}
\begin{rmk}
	At this point of the argument, we do not know that this constructed solution coincides with the one obtained in Section \ref{137}, although we will later see in Section \ref{139} that this is the case. 
\end{rmk}

We start by proving an upper bound on the growth of any smooth solution to the equation in \eqref{-113}, irrespective of the boundary data or whether $\tau<1$ or $\mu_\Gamma^+>1$:

\begin{prop}\label{101}
	Let $(M,g_0)$ be a smooth Riemannian manifold with non-empty boundary and suppose that $(f,\Gamma)$ satisfies \eqref{21'}--\eqref{24'} and \eqref{25'}. Then there exist constants $\delta>0$ and $C>0$ depending only on $g_0$ such that any continuous metric $g_u=e^{2u}g_0$ satisfying
	\begin{align}\label{102}
	f(\lambda(-g_u^{-1}A_{g_u})) \geq 1, \quad \lambda(-g_u^{-1}A_{g_u})\in\Gamma \quad   \text{in the viscosity sense on }M\backslash \partial M
	\end{align}
	satisfies
	\begin{align}\label{186}
	u(x) + \ln \operatorname{d}(x,\partial M) \leq C\operatorname{d}(x,\partial M)^{1/2} \quad \text{in }M_\delta\backslash \partial M. 
	\end{align}
	In particular, any continuous metric $g_u=e^{2u}g_0$ satisfying \eqref{102} satisfies
	\begin{align}\label{187}
	\limsup_{\operatorname{d}(x,\partial M)\rightarrow 0}\big(u(x) + \ln \operatorname{d}(x,\partial M)\big) \leq 0.
	\end{align}
	\end{prop}
\begin{proof}
	By \eqref{98}, the comparison principle for viscosity sub-/supersolutions to uniformly elliptic equations implies that if $g_w = e^{2w}g_0$ satisfies
	\begin{align}\label{99}
	\begin{cases}
	\sigma_1(-g_w^{-1}A_{g_w}) \leq \frac{n}{2} & \text{in }\Omega\Subset M\backslash \partial M \\
	w(x) \rightarrow +\infty & \text{as } \operatorname{d}(x,\partial \Omega)\rightarrow 0,
	\end{cases}
	\end{align}
	then $u\leq w$ in $\Omega$. Since $\sigma_1(-g_w^{-1}A_{g_w}) = -\frac{1}{2(n-1)}R_{g_w}$, the transformation law for scalar curvature implies that the equation in \eqref{99} is equivalent to 
	\begin{align}\label{100}
	-\frac{S_{g_0}}{n-1} + 2\Delta_{g_0} w + (n-2)|\nabla_{g_0} w|_{g_0}^2 \leq n e^{2w}. 
	\end{align}
	We follow an argument of Gursky, Streets \& Warren \cite{GSW11}, in turn based on the original argument of Loewner \& Nirenberg \cite{LN74}, to construct such local supersolutions near $\partial M$. For a point $x_0$ distance $d$ from $\partial M$, consider a point $z_0$ a distance $R>d$ from $\partial M$, which lies along the shortest path geodesic from $x_0$ to $\partial M$. We may assume $R$ is small enough so that $\Delta_{g_0} d^2(z_0, \cdot) \geq 1$ on $B_R(z_0)$, and so that there exists a function $h$ defined on $[0,R^2]$ satisfying 
	\begin{align}\label{167}
	(n-2)(h')^2 + 2h'' \leq 0, \quad h' > \operatorname{max}_M |S_{g_0}| + \widetilde{C}(g_0), \quad h(0) = 0,
	\end{align}
	where $\widetilde{C}(g_0)$ is a sufficiently large constant to be fixed in the proof. Indeed, once $\widetilde{C}(g_0)$ is fixed, the function $h(t) = \sqrt{t+\ep^2} -\ep$ satisfies \eqref{167} for $\ep$ sufficiently small and $t$ in a sufficiently small interval $[0,R^2]$. 
	
	Let $r$ denote the distance from $z_0$, and define on $B_R(z_0)$ the radial function
	\begin{align*}
	w(r) = -\ln (R^2 - r^2) + h(R^2 - r^2) + \ln \alpha, 
	\end{align*}
	where $\alpha>0$ is to be determined. Exactly as in the proof of Lemma 5.2 in \cite{GSW11}, a direct computation shows that, for $R$ sufficiently small and $\widetilde{C}(g_0)$ sufficiently large, the LHS of \eqref{100} satisfies
	\begin{align}
	-\frac{S_{g_0}}{n-1} + 2\Delta_{g_0} w + (n-2)|\nabla_{g_0} w|_{g_0}^2  & \leq \frac{4nR^2}{(R^2 - r^2)^2}e^{2h} = \frac{4nR^2}{\alpha^2}e^{2w}.
	\end{align}
	Therefore, if we take $\alpha = 2R$, we see $w$ indeed satisfies \eqref{100}. We then obtain
	\begin{align*}
	u(x_0) & \leq w(x_0) \nonumber \\
	& = -\ln (R^2 - (R-d)^2) + h(R^2 - (R-d)^2) + \ln (2R) \nonumber \\
	& = -\ln (d(2R-d)) + h(d(2R-d)) + \ln (2R) \nonumber \\
	& = -\ln  d - \ln \bigg(1-\frac{d}{2R}\bigg) + h(d(2R-d)). 
	\end{align*}
	But $h(d(2R-d)) = \sqrt{d(2R-d) + \ep^2} - \ep \leq \sqrt{d(2R-d)} \leq C\sqrt{d}$ and $\ln(1-\frac{d}{2R}) \geq -\frac{d}{2R} \geq -C\sqrt{d}$ for sufficiently small $d$, and thus \eqref{186} follows. The inequality \eqref{187} is a clear consequence of \eqref{186}.
\end{proof}

We are now in a position to prove the existence of a smooth solution to \eqref{-113} when $\tau<1$ with the desired asymptotic behaviour in \eqref{114'}.

\begin{prop}\label{105}
	Fix $\tau<1$ and suppose that $(f,\Gamma)$ satisfies \eqref{21'}--\eqref{24'}, \eqref{25'} and \eqref{418}. Then there exists a smooth solution $g_v =e^{2v}g_0$ to \eqref{-113} and a constant $C$ independent of $\tau$ but dependent on $g_0, f$ and $\Gamma$ such that the following holds: for each $\ep>0$ sufficiently small, there exists a constant $a\gg0$ independent of $\tau$ but dependent on $g_0, \ep, C, f$ and $\Gamma$ such that 
	\begin{align}\label{184}
	v(x) + \ln \operatorname{d}(x,\partial M) \geq \ln\sqrt{1-2\ep} - \ln\big(1+a\operatorname{d}(x,\partial M)\big) \quad \text{in } A^a_\ep\subset M,
	\end{align}
	where 
	\begin{align*}
	A^a_\ep = \bigg\{x\in M\backslash\partial M: d(x)+ad(x)^2 \leq \frac{\ep}{C}\bigg\}.
	\end{align*}
	In particular, 
	\begin{align}\label{103}
	\lim_{\operatorname{d}(x,\partial M)\rightarrow 0}\big(v(x) + \ln \operatorname{d}(x,\partial M)\big)  =  0.
	\end{align}
\end{prop}

\begin{proof}
	Consider an exhaustion of $M$ by smooth compact manifolds with boundary defined by $M_{(j)} = \{x\in M: \operatorname{d}(x,\partial M) \geq j^{-1}\}$. By Proposition \ref{90}, for each $j$ there exists a smooth solution $g_{v_{(j)}} = e^{2v_{(j)}}g_0$ to 
	\begin{align*}
	\begin{cases}
	f^\tau(-g_{v_{(j)}}^{-1}A_{g_{v_{(j)}}}) = 1, \quad \lambda(-g_{v_{(j)}}^{-1}A_{g_{v_{(j)}}})\in\Gamma^\tau & \text{on }M_{(j)}\backslash\partial M_{(j)} \\
	v_{(j)}(x)\rightarrow+\infty & \text{as }\operatorname{d}(x,\partial M_{(j)})\rightarrow 0. 
	\end{cases}
	\end{align*} 
	(Note that we put parentheses around the index $j$ to avoid confusion with the solutions $u_m$ to \eqref{13''}). Since $v_{(j)}(x)\rightarrow +\infty$ as $\operatorname{d}(x,\partial M_{(j)})\rightarrow 0$, the comparison principle in Proposition \ref{174} implies that if $j<m$, then
	\begin{align}\label{140}
	v_{(m)}\big|_{M_{(j)}} < v_{(j)}. 
	\end{align} 
	Now, as justified in the proof of Proposition \ref{90}, a subsequence of $\{v_{(j)}\}_j$ converges locally uniformly to some $v \in C^\infty(M\backslash\partial M)$. We claim that $v$ is our desired function. It is clear that $v$ solves the equation in \eqref{-113}. We now establish \eqref{184}, which we split into two steps: in the first step we show $v(x)\rightarrow+\infty$ as $\operatorname{d}(x,\partial M)\rightarrow 0$, and in the second step we prove \eqref{184}.\medskip
	
	\noindent\textbf{Step 1:} In this first step we show that $v(x)\rightarrow+\infty$ as $\operatorname{d}(x,\partial M)\rightarrow 0$. To this end, let $d(x) = \operatorname{d}(x,\partial M)$ and define $\phi = - \ln (B(d+ad^2))$, $g_{\phi} = e^{2\phi}g_0$, where $a$ and $B$ are positive constants to be determined. Writing $e^{2\phi} = \psi^{-2}$, so that $\psi = B(d+ad^2)$, we compute near $\partial M$
	\begin{align*}
	|\nabla_{g_0} \psi|_{g_0}^2 = B^2(1+2ad)^2 \quad \text{and} \quad  \nabla_{g_0}^2 \psi = B(1+2ad)\nabla_{g_0}^2 d + 2aB\,\nabla d \otimes \nabla d,
	\end{align*}
	where $\nabla d$ denotes the differential of $d$. It follows that near $\partial M$,
	\begin{align}\label{104}
	-g_\phi^{-1}A_{g_\phi} & = g_0^{-1}\bigg(-\psi\nabla_{g_0}^2 \psi + \frac{1}{2}|\nabla_{g_0} \psi|_{g_0}^2 g_0 - \psi^2 A_{g_0}\bigg) \nonumber \\
	& = B^2 g_0^{-1}\bigg(\frac{1}{2}g_0 + 2a^2d^2\Big[g_0 - \nabla d\otimes \nabla d - d\nabla_{g_0}^2d\Big] - d(1+3ad)\nabla_{g_0}^2 d \nonumber \\
	& \qquad \qquad + 2ad\Big[g_0 - \nabla d\otimes \nabla d\Big] - d^2(1+ad)^2A_{g_0}\bigg).
	\end{align}  
	Taking for instance $a=1$, we then see that for $\delta$ fixed sufficiently small and $B$ fixed sufficiently large, it holds that 
	\begin{align}\label{169}
	f^\tau(-g_\phi^{-1}A_{g_\phi}) \geq 1 \quad\text{in }M_\delta\backslash\partial M.
	\end{align}
	
	To use \eqref{169} to show $v(x)\rightarrow+\infty$ as $\operatorname{d}(x,\partial M)\rightarrow 0$, we follow the proof of \cite[Theorem 5]{LN74}. For $m\gg 1$, denote by $S_m$ the set where $\phi(x)= -\ln (B(d+d^2)) \geq m$. We may assume (by taking $m$ sufficiently large) that $S_m$ is a tubular neighbourhood of $\partial M$ contained in $M_\delta$. Let $\Sigma_m = \partial S_m \backslash\partial M$ and $D_m = \min_{\Sigma_m} v$, and suppose $J$ is sufficiently large so that $\Sigma_m \subset M_{(j)}$ for all $j\geq J$. Then $\phi =m$ and $v \geq D_m$ on $\Sigma_m$, and by the monotonicity in \eqref{140} we also have $v_{(j)} \geq D_m$ on $\Sigma_m$ for each $j\geq J$. Therefore
	\begin{align}\label{170}
	v_{(j)} + \operatorname{max}\{0, m-D_m\}\geq m = \phi \text{ on }\Sigma_m 
	\end{align}
	and
	\begin{align}\label{171}
	v_{(j)} + \operatorname{max}\{0, m-D_m\} = \infty > \phi \text{ on }\partial M_{(j)}. 
	\end{align}
	In light of \eqref{169}--\eqref{171}, the comparison principle in Proposition \ref{174} implies that $v_{(j)} + \operatorname{max}\{0, m-D_m\} \geq \phi$ on $M_{(j)} \cap S_m$. Sending $j\rightarrow\infty$, it follows that $v+ \operatorname{max}\{0, m-D_m\}\geq \phi$ in $S_m$, and in particular $v(x)\rightarrow+\infty$ as $\operatorname{d}(x,\partial M)\rightarrow 0$. \medskip

	\noindent\textbf{Step 2:} In this second step we show that $v$ satisfies \eqref{184}. The method is essentially a quantitative version of Step 1, requiring a more careful choice of parameters $a$ and $B$ in the definition of $\phi$.  
	
	We first claim that the two quantities in the square parentheses in \eqref{104} are nonnegative definite for sufficiently small $d$. Indeed, observe that $g_0(x) - \nabla d(x)\otimes \nabla d(x)$ is the induced metric on $\partial M_{d(x)}\backslash\partial M$, and is therefore nonnegative definite. Moreover, $\nabla_{g_0}^2 d$ is a bounded tensor near $\partial M$ whose kernel contains $\nabla d$. Hence $\nabla_{g_0}^2 d$ is bounded from above by $C(g_0 - \nabla d \otimes \nabla d)$ for some constant $C$ depending only on $(M,g_0)$. Therefore $g_0 - \nabla d\otimes \nabla d - d\nabla_{g_0}^2d$ is nonnegative definite for $d$ sufficiently small, as claimed.

	In light of \eqref{104} and the above claim, we see that for $\delta$ chosen sufficiently small independently of $a$ (but depending on $(M,g_0)$), and $\widehat{C}\geq 1$ a constant such that $|A_{g_0}|_{g_0}, |\nabla_{g_0}^2 d|_{g_0} \leq \widehat{C}$ on $M_\delta$, we have 
	\begin{align}\label{108}
	-g_\phi^{-1}A_{g_\phi} & \geq B^2 g_0^{-1}\bigg(\frac{1}{2}g_0 - d(1+3ad)\nabla_{g_0}^2 d - d^2(1+ad)^2A_{g_0} \bigg) \nonumber \\
	& \geq B^2 g_0^{-1}\bigg(\frac{1}{2} - \widehat{C}d - \widehat{C}(1+3a)d^2 - 2\widehat{C}ad^3 - \widehat{C}a^2 d^4\bigg)g_0 \quad \text{in }M_\delta\backslash\partial M.
	\end{align}
	Since we will eventually take $a$ large, we may assume $a\geq1$, in which case \eqref{108} implies
	\begin{align}\label{111}
	-g_\phi^{-1}A_{g_\phi} & \geq B^2 \bigg(\frac{1}{2} - \widehat{C}\Big[d + 4ad^2 +2ad^3 +a^2 d^4\Big]\bigg)\operatorname{Id} \quad \text{in } M_\delta\backslash\partial M.
	\end{align}
	
	Now fix $\ep>0$ small, define $B = \frac{1}{\sqrt{1-2\ep}}$ and denote by ${\widehat{A}}_\ep^{a}$ the set 
	\begin{align*}
	\widehat{A}_\ep^{a}  &= \bigg\{x\in M\backslash\partial M: \phi(x)= -\ln (B(d+ad^2)) \geq -\ln \bigg(\frac{\ep}{100\widehat{C}}\bigg) \bigg\} \nonumber \\
	& = \bigg\{x\in M\backslash\partial M: d+ad^2 \leq \frac{\ep\sqrt{1-2\ep}}{100\widehat{C}}\bigg\},
	\end{align*}
	where $\widehat{C}$ is the constant in \eqref{111}. It is easily verified that in $\widehat{A}^a_\ep$, we have $\widehat{C}(d + 4ad^2 +2ad^3 +a^2 d^4)\leq \ep$. Moreover, if we define 
	\begin{align*}
	\Sigma_\ep^{a} = \partial \widehat{A}_\ep^{a} \backslash \partial M,
	\end{align*}
	then $\Sigma_\ep^a$ converges to $\partial M$ as $a$ increases. It follows from these two facts and \eqref{111} that for $a$ sufficiently large (depending only on $(M,g_0)$),
	\begin{align*}
	-g_\phi^{-1}A_{g_\phi} & \geq B^2\operatorname{diag}\bigg(\frac{1}{2}-\ep, \dots, \frac{1}{2}-\ep\bigg) = \operatorname{diag}\bigg(\frac{1}{2},\dots,\frac{1}{2}\bigg) \quad \text{in }\widehat{A}^a_\ep. 
	\end{align*}
	It then follows from our normalisation $f(\frac{1}{2},\dots,\frac{1}{2})=1$ that 
	\begin{align}\label{110}
	f^\tau(-g_{\phi}^{-1}A_{g_\phi}) \geq 1\quad \text{in }\widehat{A}^a_\ep. 
	\end{align}

	We now let
	\begin{align*}
	\quad C_\ep^{a} = \min_{\Sigma_\ep^{a}} v. 
	\end{align*}
Since $v(x)\rightarrow +\infty$ as $\operatorname{d}(x,\partial M)\rightarrow 0$ (by Step 1), and since $\Sigma_\ep^a$ converges to $\partial M$ as $a$ increases, we can choose $a$ large enough so that $C_\ep^{a} \geq -\ln  (\frac{\ep}{100\widehat{C}})$. Moreover, this choice of $a$ depends only on $g_0, \ep, \widehat{C}, f$ and $\Gamma$: since each $v_{(j)}$ was constructed according to the procedure in the proof of Proposition \ref{90}, we know from the second statement in Proposition \ref{90} that there exists $\delta=\delta(g_0, \ep, \widehat{C}, f, \Gamma)>0$ such that $v_{(j)} \geq -\ln(\frac{\ep}{100\widehat{C}})$ in $(M_{(j)})_\delta\backslash\partial M_{(j)}$ for each $j$. Taking $j\rightarrow\infty$, we see $v \geq -\ln(\frac{\ep}{100\widehat{C}})$ in $M_\delta\backslash\partial M$. Therefore, to ensure $C^a_\ep \geq -\ln(\frac{\ep}{100\widehat{C}})$, one only needs to pick $a$ large depending on $\delta = \delta(g_0, \ep, \widehat{C}, f, \Gamma)$. 
	
	We now fix such a value of $a$ and suppose $J$ is sufficiently large so that $\Sigma^a_\ep \subset M_{(j)}$ for all $j\geq J$. Then $\phi =-\ln (\frac{\ep}{100\widehat{C}})$ and $v \geq C_\ep^{a}$ on $\Sigma_\ep^{a}$, and by the monotonicity in \eqref{140} we also have $v_{(j)} \geq C_\ep^{a}$ on $\Sigma_\ep^{a}$ for each $j\geq J$. Therefore, 
	\begin{align}\label{172}
	v_{(j)} \geq -\ln \bigg(\frac{\ep}{100\widehat{C}}\bigg) = \phi \text{ on }\Sigma_\ep^a
	\end{align}
	and
	\begin{align}\label{173}
	v_{(j)} = \infty > \phi \text{ on }\partial M_{(j)}. 
	\end{align}
	In light of \eqref{110}--\eqref{173}, the comparison principle in Proposition \ref{174} then yields
	\begin{align*}
	v_{(j)} \geq \phi \quad \text{in }\widehat{A}_\ep^{a} \cap M_{(j)}.
	\end{align*}
	Sending $j\rightarrow\infty$, it follows that $v\geq \phi$ in $\widehat{A}_\ep^a$, i.e. 
	\begin{align*}
	v \geq  \phi &  = - \ln (B(d+ad^2)) \nonumber \\
	&  =  \ln \sqrt{1-2\ep} - \ln  d - \ln (1+ad) \quad  \text{in }\widehat{A}_\ep^{a}.
	\end{align*}
	This is precisely \eqref{184} after relabelling constants, and thus the second step is complete. \medskip 
	
	To complete the proof of the proposition, we observe that \eqref{184} implies
	\begin{align*}
	\liminf_{\operatorname{d}(x, \partial M)\rightarrow 0}\big(v(x) + \ln \operatorname{d}(x,\partial M)\big) & \geq \ln \sqrt{1-2\ep}, 
	\end{align*}
	and since $\ep>0$ is arbitrary, it follows that
	\begin{align}\label{115}
	\liminf_{\operatorname{d}(x,\partial M)\rightarrow 0}\big(v(x) + \ln \operatorname{d}(x,\partial M)\big)  \geq  0.
	\end{align}
	By \eqref{115} and Proposition \ref{101}, we therefore see that $v$ satisfies \eqref{103}. 
\end{proof}
	
	\subsection{Uniqueness}\label{139}
	
	Having established the existence of a smooth solution to \eqref{-113} satisfying \eqref{114'} when $\tau<1$ and $\mu_{\Gamma}^+>1$ in the previous section, we now turn to uniqueness of solutions. We first show:

	\begin{prop}\label{144}
		Fix $\tau<1$ and suppose that $(f,\Gamma)$ satisfies \eqref{21'}--\eqref{24'}, \eqref{25'} and \eqref{418}. Then any continuous viscosity solution $g_u = e^{2u}g_0$ to \eqref{-113} satisfies \eqref{114'}. 
	\end{prop}
	\begin{proof}
		Let $u$ be a continuous viscosity solution to \eqref{-113}. By Proposition \ref{101}, we know that $u$ satisfies $\limsup_{\operatorname{d}(x,\partial M)\rightarrow 0}(u(x) + \ln \operatorname{d}(x,\partial M)) \leq 0$, so it remains to show
		\begin{align}\label{142}
		\liminf_{\operatorname{d}(x,\partial M)\rightarrow 0}\big(u(x) + \ln \operatorname{d}(x,\partial M)\big) \geq 0.
		\end{align}
		
		To prove \eqref{142}, we attach a collar neighbourhood $N$ to $\partial M$, extend $g_0$ smoothly to $M \cup N$ and consider the sequence $\{M^{(j)}\}_j$ of smooth compact manifolds with boundary given by $M^{(j)} = \{x\in M\cup N:\operatorname{d}(x,M)\leq j^{-1}\}$. Note that for $x\in M$ and $j$ sufficiently large, $\operatorname{d}(x,\partial M^{(j)}) = d(x,\partial M) + j^{-1}$. Fix $\ep>0$. By Proposition \ref{105}, there exist constants $\delta>0$ and $a>0$ depending on $g_0, \ep, f,\Gamma$ but independent of $j$, and a smooth metric $g_{u^{(j)}} = e^{2u^{(j)}}g_0$ for each $j$, such that
		\begin{align*}
		f^\tau(-g_{u^{(j)}}^{-1}A_{g_{u^{(j)}}}) = 1, \quad \lambda(-g_{u^{(j)}}A_{g_{u^{(j)}}})\in\Gamma^\tau \quad \text{on }M^{(j)}\backslash \partial M^{(j)}
		\end{align*}
		and
		\begin{align*}
		u^{(j)}(x) + \ln \operatorname{d}(x,\partial M^{(j)}) \geq \ln\sqrt{1-2\ep} -\ln \big(1+a\operatorname{d}(x,\partial M^{(j)})\big) \quad \text{in }(M^{(j)})_\delta \backslash\partial M^{(j)}.
		\end{align*}
		In particular, for $j$ sufficiently large so that $(M^{(j)})_\delta\cap M\not=\emptyset$, we have
		\begin{align}\label{420}
		u^{(j)}(x) + \ln \bigg(\operatorname{d}(x,\partial M) + \frac{1}{j}\bigg) \geq \ln\sqrt{1-2\ep} -\ln \bigg(1+a\operatorname{d}(x,\partial M) + \frac{a}{j}\bigg) \quad \text{in }M_{\delta-\frac{1}{j}}. 
		\end{align}
		
		Now, by the comparison principle in Proposition \ref{174}, $u^{(j)}|_M \leq u$ for each $j$, and thus \eqref{420} implies 
		\begin{align}\label{421}
		u(x) + \ln \bigg(\operatorname{d}(x,\partial M) + \frac{1}{j}\bigg) \geq \ln\sqrt{1-2\ep} -\ln \bigg(1+a\operatorname{d}(x,\partial M) + \frac{a}{j}\bigg) \quad \text{in }M_{\delta-\frac{1}{j}}\backslash\partial M.
		\end{align}
		After taking $j\rightarrow \infty$ in \eqref{421}, it follows that 
		\begin{align*}
		\liminf_{\operatorname{d}(x,\partial M)\rightarrow 0}\big(u(x) + \ln \operatorname{d}(x,\partial M)\big) \geq \ln\sqrt{1-2\ep},
		\end{align*}
		and since $\ep>0$ is arbitrary, we obtain \eqref{142}. 
	\end{proof}

	Finally we prove uniqueness of solutions to \eqref{-113} when $\tau<1$:

	\begin{prop}\label{422}
		Fix $\tau<1$, suppose that $(f,\Gamma)$ satisfies \eqref{21'}--\eqref{24'}, \eqref{25'} and \eqref{418}, and let $v$ denote the smooth solution to \eqref{-113} obtained in Proposition \ref{105}. Then $v$ is the unique continuous viscosity solution to \eqref{-113}.
	\end{prop}
	\begin{proof}
		Suppose that $w$ is a continuous viscosity solution to \eqref{-113}. By Proposition \ref{144}, both $v$ and $w$ satisfy \eqref{114'}. For $\delta \geq 0$, define $\Sigma_\delta = \{d=\delta\}$. Then for each $\epsilon>0$, there exists a minimal $\delta_\ep> 0$ such that $w \leq v + \ep$ on $\Sigma_{\delta_\ep}$. Denoting $v_\ep = v+\ep$, we have
		\begin{align*}
		f^\tau(-g_0^{-1}A_{g_{v_\ep}}) = f^\tau(-g_0^{-1}A_{g_v}) = e^{2v}< e^{2v_\ep}
		\end{align*}
		and thus $v_\ep$ is a supersolution the equation in \eqref{-113}. By the comparison principle in Proposition \ref{174}, it follows that $w \leq v+\ep$ on $M\backslash M_{\delta_\ep}$. By minimality of $\delta_\ep$, we have $\delta_\ep\rightarrow 0$ as $\ep\rightarrow 0$, thus $w \leq v$ on $M\backslash\partial M$. Reversing the roles of $w$ and $v$, we see also that $w \geq v$ on $M\backslash\partial M$, and therefore $w=v$. 
	\end{proof}

\subsection{Proof of Theorem \ref{A'}}\label{166}

In this final section we complete the proof of Theorem \ref{A'}:

\begin{proof}[Proof of Theorem \ref{A'}]
	The existence of a smooth solution to \eqref{-113} for each $\tau<1$, the asymptotic behaviour stated in \eqref{114} and uniqueness in the class of continuous viscosity solutions follow from Propositions \ref{105} and \ref{422}. Let us denote these solutions by $u^\tau$. As observed previously, these solutions $u^\tau$ satisfy a locally uniform $C^1$ estimate which is independent of $\tau$, i.e.~for each compact set $K\subset M\backslash \partial M$, there exists a constant $C$ independent of $\tau$ but dependent on $g_0, f, \Gamma$ and $K$ such that
\begin{align*}
\|u^\tau\|_{C^1(K)} \leq C.
\end{align*}
It follows that a subsequence of $\{u^\tau\}$ converges locally uniformly in $C^{0,\alpha}$ to some $u\in C_{\operatorname{loc}}^{0,1}(M,g_0)$ for each $\alpha\in(0,1)$. As noted in the proof of Theorem \ref{55} in Section \ref{150}, the fact that $u$ is a viscosity solution to \eqref{-113} when $\tau=1$ follows from exactly the same argument as in the proof of \cite[Theorem 1.4]{LN20b}. So it remains to show that $u$ satisfies the asymptotics in \eqref{114} and is maximal. 

To this end, first note that since we only require $u$ to be a viscosity subsolution in Proposition \ref{101}, we have
\begin{align}\label{136}
\limsup_{\operatorname{d}(x,\partial M)\rightarrow 0}\big(u(x) + \ln \operatorname{d}(x,\partial M)\big) \leq 0.
\end{align}
To show that
\begin{align}\label{155}
\liminf_{\operatorname{d}(x,\partial M)\rightarrow 0}\big(u(x) + \ln \operatorname{d}(x,\partial M)\big) \geq 0,
\end{align}
we first recall that $u$ is the $C^{0,\alpha}$ limit of the solutions $u^\tau$ as $\tau\rightarrow 1$. By Proposition \ref{105}, for each $\ep>0$ sufficiently small, there exist constants $\delta>0$ and $a>0$ independent of $\tau$ (but dependent on $g_0, \ep, f$ and $\Gamma$) such that 
\begin{align}\label{409}
u^\tau(x) + \ln \operatorname{d}(x,\partial M) \geq \ln \sqrt{1-2\ep} - \ln \big(1+a\operatorname{d}(x,\partial M)\big) \quad \text{in }M_\delta\backslash\partial M. 
\end{align}
Taking $\tau\rightarrow 1$ in \eqref{409}, we obtain 
\begin{align*}
u(x) + \ln \operatorname{d}(x,\partial M) \geq \ln \sqrt{1-2\ep} - \ln \big(1+a\operatorname{d}(x,\partial M)\big) \quad \text{in }M_\delta\backslash\partial M,
\end{align*}
and \eqref{155} then follows exactly as in the proof of Proposition \ref{105}.

Finally, to see that $u$ is maximal, suppose that $\widetilde{u}$ is another continuous viscosity solution to \eqref{-113}. By Proposition \ref{101}, \eqref{136} holds with $\widetilde{u}$ in place of $u$, and we also know that \eqref{114} is satisfied with $u^\tau$ in place of $u$ for each $\tau\leq 1$. Combining these facts, it follows that for each $\tau\leq 1$ and $\ep>0$, there exists $\delta>0$ such that $\widetilde{u} \leq u^\tau_\ep \defeq u^\tau + \ep$ in $M_\delta\backslash\partial M$. On the other hand, $f^\tau(-g_{u^\tau_\ep}^{-1}A_{g_{u^\tau_\ep}}) = e^{-2\ep} f^\tau(-g_{u^\tau}^{-1}A_{g_{u^\tau}})<1$ on $M\backslash\partial M$ and $f^{\tau}(-g_{\tilde{u}}^{-1}A_{g_{\tilde{u}}}) \geq 1$ in the viscosity sense on $M\backslash\partial M$; to see this latter inequality, observe
\begin{align*}
f^\tau(\lambda) & = \frac{1}{\tau+n(1-\tau)}f\big(\tau\lambda + (1-\tau)\sigma_1(\lambda)e\big) \nonumber \\
&  \geq  \frac{1}{\tau+n(1-\tau)} \big(\tau f(\lambda) + (1-\tau)\sigma_1(\lambda)f(e)\big) \nonumber \\
& \stackrel{\eqref{98}}{\geq}  \frac{1}{\tau+n(1-\tau)} \bigg(\tau f(\lambda) + (1-\tau)\frac{nf(\lambda)}{f(e)}f(e)\bigg)  = f(\lambda). 
\end{align*}

By the comparison principle in Proposition \ref{174}, it follows that $\widetilde{u} \leq u^\tau_\ep$ in $M\backslash M_\delta$, and therefore $\widetilde{u} \leq u_\ep^\tau$ in $M\backslash \partial M$. Taking $\ep\rightarrow 0$ and then $\tau\rightarrow 1$, it follows that $\widetilde{u} \leq u$ in $M\backslash \partial M$, as claimed. 
\end{proof}

\begin{appendices}

\section{Proof of Proposition \ref{46}: a cone property}\label{AA}

\begin{proof}[Proof of Proposition \ref{46}]
The result is essentially a consequence of \cite[Theorem 1.4]{Yuan22}. We summarise the details here for the convenience of the reader. Let $\Gamma$ be any cone satisfying \eqref{21'} and \eqref{22'}, and denote 
	\begin{equation*}
	\kappa_\Gamma = \max\{k:(\underbrace{0,\dots,0}_{k},\underbrace{1,\dots,1}_{n-k})\in\Gamma\}. 
	\end{equation*}
	Assume for now that there exists a constant $\theta=\theta(n,\Gamma)>0$ such that, whenever $\lambda\in\Gamma$ with $\lambda_1\geq \dots \geq \lambda_n$, it holds that
	\begin{equation}\label{37}
	\frac{\partial f}{\partial \lambda_i}(\lambda) \geq \theta \sum_{j=1}^n \frac{\partial f}{\partial\lambda_j}(\lambda) \quad\text{for }i \geq n-\kappa_\Gamma.  
	\end{equation}
	Since $\kappa_\Gamma=0$ if and only $\Gamma = \Gamma_n^+$, we see that $\kappa_\Gamma \geq 1$ whenever $\Gamma\not=\Gamma_n^+$, and thus \eqref{37'} holds for $i\in\{n-1,n\}$. Also, it is easy to see that $\kappa_\Gamma$ is equal to the maximum number of negative entries a vector in $\Gamma$ can have, i.e.
	\begin{align*}
	\kappa_\Gamma = \max\{k: (-\alpha_1,\dots,-\alpha_k, \alpha_{k+1}, \dots, \alpha_n)\in\Gamma,~\alpha_j>0  \text{ for all } 1\leq j \leq n\}. 
	\end{align*}
	Thus \eqref{37'} also holds if $\lambda_i\leq 0$.
	
	It remains to justify \eqref{37}, for which we follow \cite{Yuan22}. By concavity, $f_i(\lambda) \geq f_j(\lambda)$ whenever $\lambda_i \leq \lambda_j$. In particular, our ordering implies
	\begin{align*}
	\frac{\partial f}{\partial \lambda_n}(\lambda) \geq \frac{1}{n}\sum_{j=1}^n\frac{\partial f}{\partial\lambda_j}(\lambda),
	\end{align*}
	which establishes \eqref{37} for $\Gamma = \Gamma_n^+$. 
	
	On the other hand, for a general cone $\Gamma$ satisfying \eqref{21'} and \eqref{22'}, we have
	\begin{equation}\label{42}
	\sum_{i=1}^n f_i(\lambda)\mu_i>0 \quad \text{whenever }\lambda, \mu\in\Gamma.
	\end{equation}
	Suppose $\Gamma\not=\Gamma_n^+$, in which case it is clear that $\kappa_\Gamma>0$, and fix any $\alpha_1,\dots,\alpha_n>0$ such that
	\begin{equation*}
	(-\alpha_1, \dots, -\alpha_{\kappa_\Gamma}, \alpha_{\kappa_\Gamma + 1}, \dots, \alpha_n)\in\Gamma. 
	\end{equation*}
	Then \eqref{42} implies 
	\begin{align}\label{43}
	\sum_{i=\kappa_\Gamma + 1}^n \alpha_i f_{n-i+1}(\lambda) - \sum_{i=1}^{\kappa_\Gamma} \alpha_i f_{n-i+1}(\lambda)>0.
	\end{align}
	We may assume $\alpha_1 \geq \dots \geq \alpha_{\kappa_\Gamma}$, in which case \eqref{43} implies
	\begin{equation*}
	f_{n-\kappa_\Gamma }(\lambda) > \frac{\alpha_1}{\sum_{i=\kappa_\Gamma+1}^n \alpha_i} f_n(\lambda).
	\end{equation*}
	The desired estimate then follows for all $i\geq n-\kappa_\Gamma$, again by our ordering. 
\end{proof}

\section{The Schouten tensor for a radial conformal factor}\label{appb}

In this appendix we prove the formula \eqref{50}. In normal coordinates, $r = \sqrt{x_1^2 + \dots + x_n^2}$, and therefore $\partial_i v(r) = \frac{x_i}{r}v_r$. It follows that 
\begin{equation*}
|\nabla_{g_0} v|_{g_0}^2 = g_0^{ij}\partial_i v\partial_j v = \frac{g_0^{ij}x_ix_j}{r^2}v_r^2 = v_r^2, 
\end{equation*}
where we have used the fact that $\frac{\partial}{\partial r} = \frac{x_i}{\sqrt{x_1^2 + \dots + x_n^2}}\frac{\partial }{\partial x_i}$ has unit magnitude. Moreover,
\begin{align*}
(\nabla_{g_0}^2 v)_{ij} = \partial_i\partial_j v - \Gamma_{ij}^k \partial_k v = \frac{\delta_{ij}}{r}v_r + \frac{x_ix_j}{r}\bigg(\frac{v_{rr}}{r}- \frac{v_r}{r^2}\bigg) - \Gamma_{ij}^k \partial_k v. 
\end{align*}
Combining the above, we therefore see that 
\begin{align*}
(g_v^{-1}A_{g_v})^p_j & = v^2(g_0^{-1}A_{g_v})^p_j  = v^2 g_0^{pi}(A_{g_v})_{ij} \nonumber \\
& = v^2\bigg[\frac{g_0^{pi}\delta_{ij}}{vr}v_r + g_0^{pi}\frac{x_ix_j}{vr}\bigg(\frac{v_{rr}}{r}- \frac{v_r}{r^2}\bigg) - g_0^{pi}\frac{\Gamma_{ij}^kx_kv_r}{vr} - \frac{v_r^2}{2v^2}\delta_j^p + (g_0^{-1}A_{g_0})_j^p \bigg].
\end{align*}
Now write $g_0^{pi} = \delta^{pi} + \chi^{pi}$ where $\chi = O(r^2)$ as $r\rightarrow 0$. Then 
\begin{align*}
(g_v^{-1}A_{g_v})^p_j & = v^2\bigg[\frac{\delta_j^p}{vr}v_r + \frac{x^px_j}{vr}\bigg(\frac{v_{rr}}{r}- \frac{v_r}{r^2}\bigg) - \frac{v_r^2}{2v^2}\delta_j^p \bigg] \nonumber \\
& \qquad + \underbrace{v^2\bigg[\frac{\chi^{pi}\delta_{ij}}{vr}v_r + \chi^{pi}\frac{x_ix_j}{vr}\bigg(\frac{v_{rr}}{r}- \frac{v_r}{r^2}\bigg) - g_0^{pi}\frac{\Gamma_{ij}^kx_kv_r}{vr} + (g_0^{-1}A_{g_0})_j^p\bigg]}_{= \Psi_j^p} \nonumber \\
& = v^2\bigg(\lambda\delta_j^p + \chi\frac{x^px_j}{r^2}\bigg) + \Psi_j^p,
\end{align*}
where $\lambda$ and $\chi$ are as in \eqref{16}. Now, since $\chi = O(r^2)$ we have
\begin{equation*}
v^2\frac{\chi^{pi}\delta_{ij}}{vr}v_r = O(r)v|v_r| \quad \text{and}\quad v^2\chi^{pi}\frac{x_ix_j}{vr}\bigg(\frac{v_{rr}}{r}-\frac{v_r}{r^2}\bigg) = O(r^2)v|v_{rr}| + O(r)v|v_r|,
\end{equation*}
and since $\Gamma_{ij}^k = O(r)$ and $(g_0^{-1}A_{g_0})^p_j = O(r)$, we also have
\begin{equation*}
v^2g_0^{pi}\frac{\Gamma_{ij}^kx_kv_r}{vr} = O(r)v|v_r| \quad\text{and}\quad v^2(g_0^{-1}A_{g_0})_j^p = O(r)v^2. 
\end{equation*}
The claim \eqref{50} then follows. 
\end{appendices}

\renewcommand{\baselinestretch}{0.9}
\small
\bibliography{references}{}
\bibliographystyle{siam}

	\end{document}